\documentclass[12pt]{amsart}

\usepackage{amssymb}
\usepackage{amsthm}
\usepackage{amsrefs}
\usepackage{mathtools}
\usepackage{stmaryrd} % llbracket
\usepackage{comment}
\newtheorem{thm}{Theorem}[section]
\newtheorem{prop}[thm]{Proposition}
\newtheorem{lem}[thm]{Lemma}
\newtheorem{cor}[thm]{Corollary}

\newtheorem{rem}[thm]{Remark}
\newcommand{\R}{\mathbb{R} }
\newcommand{\Q}{\mathbb{Q}}
\newcommand{\Z}{\mathbb{Z}_{>0}}
\newcommand{\Hoff}{\mathfrak{H}}

\newcommand{\D}{\mathfrak{D}}

\newcommand{\Hhat}{\hat{\Hoff}}
\allowdisplaybreaks[1]
\begin{document}
%
% タイトル subtitleは指定しない場合は表示されない
\title[Intersection of duality and derivation relations for multiple zeta values]{Intersection of duality and derivation relations\\ for multiple zeta values}
%
% 著者一覧 
\author{Aiki Kimura}
\begin{abstract}
    The duality relation is a basic family of linear relations for multiple zeta values. The extended double shuffle relation (EDSR) is one of the families of relations expected to generate all linear relations among multiple zeta values, but it remains unclear as to whether all duality relations can be deduced from the EDSR. 
    In the present paper, regarding the family generated by the duality relation and the family generated by the derivation relation, an explicit characterization of their intersection is obtained. Here, the derivation relation is a specialization of the EDSR. 
    %多重ゼータ値の代数の構造解明において, 双対関係式は基本的かつ重要な関係式族のひとつである. 多重ゼータ値のすべての $\Q$-線形関係式を生成する, と期待される一般複シャッフル関係式について, 双対関係式の生成可能性は未解明の重要課題である. 
    %本稿では, 双対関係式と一般複シャッフル関係式の特殊化である導分関係式について, それぞれが生成する関係式族の共通部分に特徴づけを与えたので報告する. 
\end{abstract}
%
% キーワード
% AMS分類 
 %\sep 16W60 \sep 06A07
% multiple zeta,  formal power series, ordered set
\maketitle
%
%%%%%%%%%%%%%%%%%%%%%%%%%%%%%%%%%%%%%%%%%%
\section{Introduction}
%%%%%%%%%%%%%%%%%%%%%%%%%%%%%%%%%%%%%%%%%%

% 正整数の組 $(k_1, \dots, k_r)$ をインデックス, $k_1 \geq 2$ であるものを収束インデックスと呼ぶ. 収束インデックス $(k_1, \dots, k_r)$ に対して, 多重ゼータ値を次の収束多重級数で定義する: 
%We call a tuple of positive integers an \textit{index}, and call the one with $k_1 > 1$ an \textit{admissible index}. For an admissible index $(k_1, \dots, k_r)$, we define the multiple zeta value by 
Multiple zeta values (MZVs) are defined by the convergent series
\[
    \zeta(k_{1}, \dots, k_{r}) := \sum_{m_1>\dots>m_r>0}\frac{1}{ m^{k_1}_{1}\cdots m^{k_r}_{r} }
\]
for positive integers $k_1, \ldots, k_r$ with $k_1 \geq 2$. 
% インデックス $(k_1, \dots, k_r)$ に対して, $k_1 + \cdots +k_r$ を重さ, $r$ を深さ, $2$以上である成分の個数を高さと呼ぶ. 
For an \textit{index} $(k_{1}, \dots, k_{r})$, we call $k_1 + \cdots +k_r$ \textit{weight}, $r$ \textit{depth}, and $\#\{k_i\, \vert\,k_i \geq 2\}$ \textit{height}. 
%
% 任意の整数 $k \geq 2$ に対して, $\Zeta_k$ を重さ $k$ の多重ゼータ値が張る $\Q$-ベクトル空間とする. 数列 $\{d_k\}_{k_\geq0}$ を漸化式 $d_k = d_{k-2} + d_{k-3}$, 初期値 $d_0=1$, $d_1 = 0$, $d_2=1$ によって定めると, Zagierによって, 
%
Regarding the $\Q$-linear space spanned by all MZVs, there exist many $\Q$-linear relations among MZVs. 
Moreover, this $\Q$-linear space has an algebraic structure, but its structure remains unexplained. 
The duality relation is a basic and important family of $\Q$-linear relations for MZVs. 
In \cite{IKZ}, the extended double shuffle relation (EDSR) is also one of the families of relations expected to generate all $\Q$-linear relations among MZVs. 
%
% しかしながら, 一般複シャッフル関係式についての双対関係式の生成可能性は, いまだ未解明である．
However, it remains unclear as to whether all duality relations can be deduced from the EDSR. 

We consider the family generated by the duality relation and the family generated by the derivation relation. Here, the derivation relation is a specialization of the EDSR. 
Our first goal is to characterize the intersection of these families explicitly. 
Based on this characterization, we can generate all relations of the intersection. 
In particular, we present the four identities as special cases. From two of these cases, we obtain new proofs of the results of Kajikawa \cite{Kaji} and Li \cite{Li}, and from the other two identities, it is newly pointed out that the two families of the duality relations can be deduced from the derivation relation. 
Note that all of the duality relations cannot be deduced from the intersection. This fact is checked by numerical calculation (see \cite{KT} for details). 

% 目次
The remainder of the present article is organized as follows. 
In Section \ref{sec:preparation}, we review the basic terminology and well-known results regarding the duality and the derivation relations. 
In Section \ref{sec:characterization}, we first state the key equation on the intersection of the duality and the derivation relations (Theorem \ref{thm:Ker}). Next, we see that the key equation characterizes the entire intersection explicitly (Theorem \ref{thm:dual cap deri}). 
In Section \ref{sec:explicit identity}, we present the four identities as special cases of Theorem \ref{thm:Ker} (Corollary \ref{cor:explicit}). 
In Section \ref{sec:proof}, we prove Theorems \ref{thm:Ker} and \ref{thm:dual cap deri}. 
In Appendix \ref{sec:equivalence}, we show that Corollary \ref{cor:explicit} (i) and (iii) are equivalent to the results of Kajikawa \cite{Kaji} and Li \cite{Li}, respectively. We also mention the results obtained by Kawasaki and Tanaka in Remark \ref{rem:KT}. 
In Appendix \ref{sec:some properties}, we also show a property related to the characterization (Corollary \ref{cor:cap}). 

%%%%%%%%%%%%%%%%%%%%%%%%%%%%%%%%%%%%%%%%%%%%%%%%%%%%%
\section{Preparation} \label{sec:preparation}
%%%%%%%%%%%%%%%%%%%%%%%%%%%%%%%%%%%%%%%%%%%%%%%%%%%%%

% {\cite{Hoffman}}でHoffmanによって定義された, 有理数体上の非可換多項式環 $\Hoff :=\Q \langle x, y \rangle$, 
% およびその部分環 $\Hoff^0 := \Q + x \Hoff y$ を導入する. 
To state our results, we use the algebraic setup introduced by Hoffman \cite{Hoffman}. Let $\Hoff =\Q \langle x, y \rangle$ be the non-commutative polynomial algebra over the rationals in two indeterminates $x$ and $y$, and let $\Hoff^0$ be its subalgebra $\Q + x \Hoff y$. 
%
% 写像 $Z: \Hoff^0 \rightarrow \R$ を $\Q$-線形であるとして, 
%Reviewing that $\Hoff^0$ is spanned by $x^{k_{1}-1}y \cdots x^{k_{r}-1}y$ with $k_1 > 1$ over $\Q$, 
We define the $\Q$-linear map $Z: \Hoff^0 \rightarrow \R$ by $Z(1) = 1$ and
\[
    Z(x^{k_{1}-1}y \cdots x^{k_{r}-1}y) = \zeta (k_1, \cdots, k_r) \quad (k_1 \geq 2). 
\]
% 重さと深さがそれぞれ $x, y$ の次数の和と $y$ の次数に対応しており, 高さは, 単項式を構成する $x^m y^n$ ($m,n \geq 1$) の個数に一致する. 
For an index $(k_1, \ldots, k_r)$, the weight $k$ and the depth $r$ correspond to the total degree and the degree in $y$ of the monomial $x^{k_{1}-1}y \cdots x^{k_{r}-1}y$, respectively. 
The height also corresponds to the number of $xy$ in the monomial. 
To obtain a $\Q$-linear relation among MZVs is precisely to obtain an element in the kernel of the map $Z$.
% 反自己同型写像 $\tau:\Hoff \rightarrow \Hoff$を $\tau(x) = y$, $\tau(y)=x$ で定義する. 
Let $\tau:\Hoff \rightarrow \Hoff$ be the anti-automorphism of the algebra $\Hoff$ defined by 
\[
    \tau(x) = y, \quad \tau(y)=x.
\]
The map $\tau$ is an involution and preserves $\Hoff^0$. 
% 任意の $w \in \Hoff^0$ に対して, 
Then, the duality relation is stated as
\[
    (\tau -\mathrm{id}_{\Hoff})(w) \in \mathrm{Ker}(Z) 
\]
for any $w \in \Hoff^0$ (see, e.g., \cite{Zagier}). 
% が成り立ち, 双対関係式としてよく知られている. 
%
% また, 正整数 $n$ に対して, $\Q$-線形写像 $\partial_n : \Hoff \rightarrow \Hoff$ をLeibniz則に従うとして, 以下で定義する: 
For a positive integer $n$, we define the $\Q$-linear map $\partial_n : \Hoff \rightarrow \Hoff$ by 
\[
    \partial_n(x) = -\partial_n(y) = x(x + y)^{n-1}y
\]
and the Leibniz rule
\[
    \partial_n(w_1 w_2) = \partial_n(w_1) w_2 + w_1 \partial_n(w_2), 
\] 
where $w_1, w_2 \in \Hoff$. 
% ここで, $w_1, w_2 \in \Hoff$ である. 
We find in \cite{IKZ} that $\partial_n(\Hoff^0) \subset \Hoff^0$ and that $\mathrm{deg}(\partial_n(w))=\mathrm{deg}(w)+n$ for any monomial $w \in \Hoff$. We also find that $\partial_n$ and $\partial_m$ commute for any $n, m \geq 1$. 
The derivation relation, which is obtained by Ihara, Kaneko, and Zagier  \cite[Theorem 3]{IKZ}, is stated as
\[
    \partial_n(w) \in \mathrm{Ker}(Z) 
\]
for any positive integer $n$ and $w \in \Hoff^0$. 
% が成り立ち, 導分関係式\cite{IKZ}*{Corollary 6}と呼ばれる. 
% \cite{IKZ}*{Chapter 6}において, 任意の正整数 $m, n$ に対して, $\partial_m$ と $\partial_n$ は可換であることが述べらている. 
%
% 不定元 $u$ に対して, $\Hoff\llbracket u\rrbracket $ を $\Hoff$ 上の形式的べき級数環とする. 
Let $\Hoff\llbracket u\rrbracket $ be the formal power series ring generated by the indeterminate $u$ over $\Hoff$. 
% 写像 $\tau$, $\partial_n$ について, $\tau(u) = u$, $\partial_n(wu) = \partial_n(w)u$ ($w \in \Hoff\llbracket u\rrbracket $) と定めることで, $\Hoff\llbracket u\rrbracket $ 上に拡張される. 
By extending $\tau$ and $\partial_n$ to $\Q\llbracket u \rrbracket$-linear maps, $\tau$ and $\partial_n$ both become maps on $\Hoff\llbracket u \rrbracket$. 
% \cite{IKZ}に従って, 自己同型写像 $\Delta_u : \Hoff\llbracket u\rrbracket  \rightarrow \Hoff\llbracket u\rrbracket $ を
Let $\Delta_u$ be the map on $\Hoff \llbracket u \rrbracket$ defined by
\[
    \Delta_u = \mathrm{exp} \left( \sum_{n \geq 1} \frac{\partial_n}{n} u^{n} \right). 
\]
Then, $\Delta_{u}$ is an automorphism of $\Hoff\llbracket u\rrbracket $ and satisfies 
% で定義すると, $\Delta_u$は以下を満たす: 
\[
    \Delta_u(x) = x \frac{1}{1-yu}, \quad
    \Delta_u(y) = (1-xu-yu) \frac{y}{1-yu}, \quad
    \Delta_u(x+y) = x+y, 
\]
% 逆写像 ${\Delta_u}^{-1}$ は以下を満たす: 
where $1/(1-wu) := \sum_{i \geq 0} w^i u^i$ ($w \in \Hoff\llbracket u\rrbracket $). (See \cite[Corollary 3 and Theorem 4]{IKZ} for details.)
The inverse map ${\Delta_u}^{-1}$ also satisfies 
\[
    {\Delta_u}^{-1}(x) = \frac{x}{1-xu} (1-xu-yu), \quad
    {\Delta_u}^{-1}(y) = \frac{1}{1-xu} y, \quad
    {\Delta_u}^{-1}(x+y) = x+y. 
\]
% ここで, $\frac{1}{1-wu} := \sum_{i \geq 0} w^i u^i (= (1-wu)^{-1})$ ($w \in \Hoff\llbracket u\rrbracket $) である. 
% 不定元 $u$ に対して, $\Hoff^0\llbracket u\rrbracket $ を $\Hoff^0$ 上の形式的べき級数環とすると, 写像 $\tau$, $\Delta_{u}$ はそれぞれ $\Hoff^0\llbracket u\rrbracket $ を保存する. 
% 写像 $\Delta_{u}$ と ${\Delta_{u}}^{-1}$ について, それぞれ $e$ 回合成したものを ${\Delta_u}^e := \underbrace{\Delta_u \circ \cdots \circ \Delta_u}_{e}$\,, 
We denote by ${\Delta_u}^e := \underbrace{\Delta_u \circ \cdots \circ \Delta_u}_{e}$, 
${\Delta_u}^{-e} := \underbrace{{\Delta_u}^{-1} \circ \cdots \circ {\Delta_u}^{-1}}_{e}$, 
% で表し, ${\Delta_u}^{0} := \mathrm{id} = \mathrm{id}|_{\Hoff\llbracket u\rrbracket }$ とする. 
and ${\Delta_u}^{0} := \mathrm{id} = \mathrm{id}_{\Hoff\llbracket u\rrbracket }$. 
% 写像 $\Delta_{u}$ における $u^i$ ($i \in \Z$) の係数を $\theta_i$ とおく. 
% すなわち, $\Delta_u = \mathrm{id} + \sum_{i \geq 1} \theta_i u^i$. 
For a positive integer $i$, let $\theta_i$ be the coefficient of $u^i$ in the expansion of $\Delta_{u}$, that is, $\Delta_u = \mathrm{id} + \sum_{i \geq 1} \theta_i u^i$. 
% 各 $\theta_i$ は $\Q[\partial_1, \partial_2, \ldots]$ の元である. このとき, $\theta_1 = \partial_1$ であることと, 一般に
Then, it is easy to see that each $\theta_i$ is composed of the maps $\partial_n$. 
\begin{thm}[Ihara-Kaneko-Zagier {\cite[Theorem 3]{IKZ}}] \label{thm:derivation}
    The following properties are equivalent: 
    \begin{enumerate}
        \item[(i)] $\partial_n(w) \in \mathrm{Ker}(Z)$ for all $w \in \Hoff^0$ and $n \in \Z$; 
        \item[(ii)] $\theta_i(w) \in \mathrm{Ker}(Z)$ for all $w \in \Hoff^0$ and $i \in \Z$. 
    \end{enumerate}
\end{thm}

%%%%%%%%%%%%%%%%%%%%%%%%%%%%%%
\section{Characterization} \label{sec:characterization}
%%%%%%%%%%%%%%%%%%%%%%%%%%%%%

% 主結果として, 双対関係式と導分関係式がそれぞれ生成する関係式族の共通部分に特徴づけを与える. 
In this section, considering the family generated by the duality relation and the family generated by the derivation relation, we state the main results for their intersection. First, we see the key equation on the intersection of the duality and derivation relations (Theorem \ref{thm:Ker}). Second, we explicitly characterize the entire intersection (Theorem \ref{thm:dual cap deri}). The proofs of both theorems are given in Section \ref{sec:proof}. 

Fix a positive \mbox{integer $s$}. For the indeterminates $u_1, \ldots, u_s$, let $\Hoff\llbracket u_1, \ldots, u_s \rrbracket$ be the commutative formal power series ring generated by $u_1, \ldots, u_s$ over $\Hoff$. 
% 非可換多項式環 $\Hoff$ 上で不定元 $u_1, \ldots, u_s$ によって生成される形式的べき級数環を $\Hoff\llbracket u_1, \ldots, u_s\rrbracket $ とおく．
% 任意の整数 $i, j$ ($1 \leq i, j \leq s$) に対して, $\Delta_{u_i}(u_j) = u_j$ と定めると, 写像 $\Delta_{u_i}$ は形式的べき級数環 $\Hoff\llbracket u_1, \ldots, u_s\rrbracket $ の自己同型写像に拡張される. 
By setting $\Delta_{u_i}(u_j) = u_j$ for $i, j$ ($1 \leq i, j \leq s$), each $\Delta_{u_1}, \ldots, \Delta_{u_s}$ is extended to an automorphism of $\Hoff\llbracket u_1, \ldots, u_s\rrbracket$. 
% H上で可換
\begin{prop} \label{prop:commutative}
    %形式的べき級数環 $\Hoff\llbracket u_1, \ldots, u_s\rrbracket $ 上で, $\Delta_{u_i}$ と $\Delta_{u_j}$ は可換である. 
    On $\Hoff\llbracket u_1, \ldots, u_s\rrbracket $, the maps $\Delta_{u_i}$ and $\Delta_{u_j}$ are commutative for $i$ and $j$ ($1 \leq i, j \leq s$). 
\end{prop}
\begin{proof}
    % 写像 $\partial_n$ と $\partial_m$ は可換であるので, 定義から
    The maps $\partial_n$ and $\partial_m$ are commutative for any positive integers $n$ and $m$. Therefore, this assertion follows from the definitions of $\Delta_{u_i}$ and $\Delta_{u_j}$. 
\end{proof}
%
% 整数 $e_1, \ldots, e_s$ に対して, 
% 関数の合成 $\Delta = {\Delta_{u_1}}^{e_1} \circ \cdots \circ {\Delta_{u_s}}^{e_s}$ 
% は $\Hoff\llbracket u_1, \ldots, u_s\rrbracket $ の自己同型写像である. 
By setting $\tau(u_i)= u_i$ ($1 \leq i \leq s$), the map $\tau$ is also extended to an anti-automorphism of $\Hoff\llbracket u_1, \ldots, u_s\rrbracket $. 
For an automorphism $\Delta = {\Delta_{u_1}}^{e_1} \circ \cdots \circ {\Delta_{u_s}}^{e_s}$, we define the $\Q$-linear space $\D_{\Delta}$ by
% \footnote{The notation $\D$ comes from the initial letter of the Duality and Derivation relations. }
\[
    \D_{\Delta} := \mathrm{span}_{\Q} \left\{ a b \tau(\Delta(a)) \; \bigm\vert \; a \in \Hoff\llbracket u_1, \ldots, u_s\rrbracket , \; b \in (\Q[z])\llbracket u_1, \ldots, u_s\rrbracket  \right\}, 
\]
where $z = x+y$. 
By considering the case $a=1$, it is easy to see that $\D_\Delta$ includes $(\Q[z])\llbracket u_1, \ldots, u_s\rrbracket$, so $\D_\Delta$ is not an empty set. 
% 併せて $\D_{\Delta}^0$ を $\D_{\Delta} \cap \Hoff^0\llbracket u_1, \ldots, u_s\rrbracket $ とおく. 
Let $\D_{\Delta}^0$ denote $\D_{\Delta} \cap \Hoff^0\llbracket u_1, \ldots, u_s\rrbracket $. 
\begin{thm}\label{thm:Ker}
    % ある $w \in \Hoff\llbracket u_1, \ldots, u_s\rrbracket $ が
    A certain element $w \in \Hoff\llbracket u_1, \ldots, u_s\rrbracket$ satisfies the following equation if and only if $w \in \D_{\Delta}$: 
    \begin{equation} \label{eqn:DD}
        (\tau -\mathrm{id})(w) = (\Delta -\mathrm{id})(w). 
    \end{equation}
    % を満たすならば, $w \in \D_{\Delta}$ に限る. 
\end{thm}
We give the proof of Theorem \ref{thm:Ker} in Section \ref{sec:proof}. 
\begin{rem}\upshape \label{rem:identity}
    % 特に$w \in \D_{\Delta}^{0}$ ならば, 左辺は双対関係式に, 右辺は導分関係式にそれぞれ対応し, $(\tau - \mathrm{id})(w)$ に対応する双対関係式が右辺の導分関係式から明示的に生成できることを意味する. 
    In Equation \eqref{eqn:DD}, if $w \in \Hoff^0 \llbracket u_1, \ldots, u_s\rrbracket$, i.e., $w \in \D_{\Delta}^0$, then the left- and right-hand sides coincide with the duality and derivation relations, respectively. In other words, the duality relation of the left-hand side is generated by the derivation relation of the right-hand side, explicitly. 
\end{rem}
\begin{prop} \label{prop:power}
    % 写像$\Delta$に対して, 
    % $w \in \D_{\Delta}$ ならば, 任意の正整数 $d$ で $w^d \in \D_{\Delta}$. 
    For any $\Delta$, 
    if $w \in \D_{\Delta}$, then $w^d \in \D_{\Delta}$ ($d \in \Z$). 
\end{prop}
\begin{proof}
    %任意の $w \in \D_\Delta$ について, 定理$\ref{thm:Ker}$から $\tau(w) = \Delta(w)$ である. 
    For any $w \in \D_\Delta$, it is easy to see $\tau(w) = \Delta(w)$ from Theorem \ref{thm:Ker}. 
    %このとき, 
    Then, we have 
    \[
        \tau(w^d) = \tau(w)^d = \Delta(w)^d = \Delta(w^d). 
    \]
    Therefore, we also obtain $(\tau-\mathrm{id})(w^d) = (\Delta-\mathrm{id})(w^d)$, and so this proof is completed using Theorem \ref{thm:Ker}. 
\end{proof}
%
% 形式的べき級数環 $\Hoff\llbracket u_1, \ldots, u_s\rrbracket $ の部分集合 $\mathfrak{A}$ に対して, $\mathrm{Coef}_{u_1, \ldots, u_s}(\mathfrak{A})$ を $\mathfrak{A}$ の係数全体の集合として次で定義する: 
For a subset $\mathfrak{A}$ of $\Hoff\llbracket u_1, \ldots, u_s\rrbracket $, let $\mathrm{Coef}_{u_1, \ldots, u_s}(\mathfrak{A})$ denote the set that is composed of all of the coefficients $w_{i_1, \ldots, i_s}$ appearing in each element $\sum_{i_1, \ldots, i_s \geq 0}w_{i_1, \ldots, i_s} u_1^{i_1} \cdots u_s^{i_s} \in \mathfrak{A}$. 
% 部分集合 $\mathfrak{A}$ が $\Q$-部分ベクトル空間であれば, $\mathrm{Coef}_{u_1, \ldots, u_s}(\mathfrak{A})$ も $\Q$-部分ベクトル空間である. 
If the subset $\mathfrak{A}$ is a $\Q$-linear space, then $\mathrm{Coef}_{u_1, \ldots, u_s}(\mathfrak{A})$ is a $\Q$-linear space, too. 
% 集合 $\mathfrak{A}$ に対して, $\partial(\mathfrak{A})$ を $\mathrm{span}_{\Q}\{\partial_n(w) \; | \; n \in \Z, \; w \in \mathfrak{A}\}$ とおく. 
Let $\partial(\mathfrak{A})$ be the $\Q$-linear space 
$\mathrm{span}_{\Q}\{\partial_n(w) \; | \; n \in \Z, \; w \in \mathfrak{A}\}$. 
\begin{thm}[] \label{thm:dual cap deri}
    % 形式的べき級数環 $\Hoff \llbracket u_1, \ldots, u_s\rrbracket $ において, $\Delta = {\Delta_{u_1}}^{e_1} \circ \cdots \circ {\Delta_{u_s}}^{e_s}$ ($e_1, \ldots, e_s$ は整数) とする. 
    For integers $e_1, \ldots, e_s$ with $(e_1, \ldots, e_s) \neq (0, \ldots, 0)$, let $\Delta = {\Delta_{u_1}}^{e_1} \circ \cdots \circ {\Delta_{u_s}}^{e_s}$. 
    % このとき, $(e_1, \ldots, e_s) \neq (0, \ldots, 0)$ ならば, 以下の3つの集合はすべて等しい: 
    Then, the following sets are equal:
    \begin{enumerate}
        \item[(i)] $(\tau -\mathrm{id})(\Hoff) \cap \partial(\Hoff)$; 
        \item[(ii)] $\mathrm{Coef}_{u_1, \ldots, u_s} \left((\tau -\mathrm{id})(\D_{\Delta}) \right)$; 
        \item[(iii)] $\mathrm{Coef}_{u_1, \ldots, u_s} \left((\Delta -\mathrm{id})(\D_{\Delta}) \right)$. 
    \end{enumerate}
\end{thm}
We give the proof of Theorem \ref{thm:dual cap deri} in Section \ref{sec:proof}. 
This theorem still remains true even if we restrict to $\Hoff^0$. 
\begin{cor} \label{cor:H0}
    For integers $e_1, \ldots, e_s$ with $(e_1, \ldots, e_s) \neq (0, \ldots, 0)$, let $\Delta = {\Delta_{u_1}}^{e_1} \circ \cdots \circ {\Delta_{u_s}}^{e_s}$. 
    Then, the following sets are equal: 
    \begin{enumerate}
        \item[(i)] $(\tau -\mathrm{id})(\Hoff^0) \cap \partial(\Hoff^0)$; 
        \item[(ii)] $\mathrm{Coef}_{u_1, \ldots, u_s} \left((\tau -\mathrm{id})(\D_{\Delta}^0) \right)$; 
        \item[(iii)] $\mathrm{Coef}_{u_1, \ldots, u_s} \left((\Delta -\mathrm{id})(\D_{\Delta}^0) \right)$. 
    \end{enumerate}
\end{cor}
\begin{rem}\upshape \label{rem:intersection}
    % 特に, 上の集合を$\Hoff^0$ に制限したとき, (i)の集合は双対関係式と導分関係式の共通部分に対応して, (ii)および(iii)に等しいことから, 共通部分全体を集合 $\D_{\Delta}^0$ で生成できることがわかる. すわなち, 集合 $\D_{\Delta}^0$ は双対関係式と導分関係式の共通部分に特徴づけを与えている. 
    Corollary \ref{cor:H0} (i) coincides exactly with the entire intersection of the duality and derivation relations. Since both Corollary \ref{cor:H0} (ii) and (iii) are stated by the explicit set $\D_{\Delta}^{0}$, each of these explicitly characterizes the intersection. 
    Furthermore, all relations of the intersection can be generated by substituting elements of $\D_{\Delta}^{0}$ into Equation \eqref{eqn:DD}. 
\end{rem}
Note that Theorem \ref{thm:dual cap deri} and Corollary \ref{cor:H0} hold for any positive integer $s$. It suffices to consider only the case in which $s=1$.  However, in Section \ref{sec:explicit identity}, we explain that the multiplex indeterminates $u_1, \ldots, u_s$ are useful to fix the parameters. 

\section{Explicit Identities} \label{sec:explicit identity}
% 定理\ref{thm:Ker}から得られる特殊な例として, 共通部分において成り立つ具体的な等式を紹介する: 
As special cases of Theorem \ref{thm:Ker}, we show the four explicit identities on the intersection of the duality and derivation relations. From two of the four identities, we reobtain the results of Kajikawa \cite{Kaji} and Li \cite{Li}, respectively. 
From the other two identities, it is newly pointed out that the derivation relation can generate two families of the duality relations for the special indices. 
Note that these results are obtained due to the extension of the multiple indeterminates $u_1, \ldots, u_s$. 
 
\begin{cor}[]
    % 正整数 $d$ に対して, 以下が成り立つ: \label{cor:explicit}
    Assume $s \geq 3$ and $d$ is a positive integer. Then, we have the following: \label{cor:explicit}
    \begin{enumerate}
    \item[(i)]
        $\displaystyle(\tau - \mathrm{id}) \left( \Bigl( \frac{x}{1-xu_1} \frac{y}{1-yu_2} \Bigr)^d \right)
        =(\Delta_{u_1} \circ \Delta_{u_2}^{-1} -\mathrm{id}) \left( \Bigl(\frac{x}{1-xu_1} \frac{y}{1-yu_2} \Bigr)^d \right). $
    \item[(ii)] 
        $\displaystyle(\tau - \mathrm{id}) \left( \Bigl( x\frac{1}{1-yu_1} \frac{1}{1-xu_2}y \Bigr)^d \right)
        = (\Delta_{u_1}^{-1} \circ \Delta_{u_2} -\mathrm{id}) \left( \Bigl(x\frac{1}{1-yu_1} \frac{1}{1-xu_2}y \Bigr)^d \right). $
    \item[(iii)]
        $\displaystyle (\tau - \mathrm{id}) \left( \frac{xu_1}{1-xu_1} y \frac{1}{1- \frac{1}{1-xu_3}yu_2} \frac{1}{1-xu_3}y \right)$ \\
        $\displaystyle = -(\Delta_{u_1} \circ \Delta_{u_2}^{-1} \circ \Delta_{u_3} -\mathrm{id}) \left( \frac{x}{1-xu_1}(1-xu_1-yu_1) \frac{1}{1- \frac{1}{1-xu_3}yu_2} \frac{1}{1-xu_3}y \right)$ \\ 
        \hfill $\displaystyle + (\Delta_{u_2}^{-1} \circ \Delta_{u_3} -\mathrm{id}) \left( x \frac{1}{1- \frac{1}{1-xu_3}yu_2} \frac{1}{1-xu_3}y \right). $
    \item[(iv)]
        $\displaystyle (\tau - \mathrm{id}) \left( x\frac{1}{1-yu_1} \frac{1}{1- xu_2 \frac{1}{1-yu_1}}y \frac{1}{1-xu_3}yu_3 \right)$ \\
        $\displaystyle = -(\Delta_{u_1}^{-1} \circ \Delta_{u_2} \circ \Delta_{u_3}^{-1} -\mathrm{id}) \left( x\frac{1}{1-yu_1} \frac{1}{1- xu_2 \frac{1}{1-yu_1}} (1-xu_3-yu_3) \frac{1}{1-xu_3}y \right)$ \\
        \hfill $\displaystyle+ (\Delta_{u_1}^{-1} \circ \Delta_{u_2} -\mathrm{id}) \left( x \frac{1}{1-yu_1} \frac{1}{1- xu_2 \frac{1}{1-yu_1}}y \right). $\\
    \end{enumerate}
\end{cor}
\begin{proof}[Proof of Corollary \ref{cor:explicit}]
    % 条件式\eqref{eqn:DD}を満たす元として, 特別な $\D_{\Delta}^0$ の元を選ぶことでそれぞれの等式を与える. 
    Substituting one or two special elements of $\D_{\Delta}^0$ into Equation \eqref{eqn:DD}, we obtain each identity. 
    % (i)について, $\Delta = \Delta_{u_1}\circ {\Delta_{u_2}}^{-1}$, $a={\Delta_{u_1}}^{-1}(x)$, $b=(1-xu_1-yu_1)^{-1} (1-xu_2-yu_2)^{-1}$ とすると, 命題$\ref{prop:power}$から
    For the identity (i), we take the special element $ab\tau(\Delta(a))$, where $\Delta = \Delta_{u_1}\circ {\Delta_{u_2}}^{-1}$, $a={\Delta_{u_1}}^{-1}(x)$ and $b=(1-xu_1-yu_1)^{-1} (1-xu_2-yu_2)^{-1}$. 
    From Proposition \ref{prop:power}, we see that $\left\{ a b \tau(\Delta(a)) \right\}^d \in \D^{0}_{\Delta}$, i.e., 
    \[
        \left(x \frac{1}{1-xu_1} \frac{1}{1-yu_2}y\right)^d \in \D^{0}_{\Delta}. 
    \]
    % がわかる. 
    % 定理$\ref{thm:Ker}$から, この元は条件式\eqref{eqn:DD}を満たすので, (i)の等式が得られる. 
    By substituting the above element into Equation \eqref{eqn:DD}, the identity (i) is obtained. 
    %
    % (ii)についても同様に, $\Delta = {\Delta_{u_1}}^{-1} \circ \Delta_{u_2}$, $a=\Delta_{u_1}(x)$, $b=1$ とすると, 命題$\ref{prop:power}$から
    Similarly, for the identity (ii), we also take a special element $ab\tau(\Delta(a))$, where $\Delta = {\Delta_{u_1}}^{-1} \circ \Delta_{u_2}$, $a=\Delta_{u_1}(x)$, and $b=1$. 
%     From Proposition \ref{prop:power}, we see
%    \[
%        \left(x \frac{1}{1-yu_1} \frac{1}{1-xu_2}y\right)^d = \left\{ a b \tau(\Delta(a)) \right\}^d \in \D^{0}_{\Delta}. 
%    \]
    % がわかる. 
    % 定理$\ref{thm:Ker}$から, この元は条件式\eqref{eqn:DD}を満たすので, (ii)の等式が得られる. 
    By substituting $\{ab\tau(\Delta(a))\}^d$ into Equation \eqref{eqn:DD}, the identity (ii) is obtained, too. 
    %
    % (iii)について, $\Delta = \Delta_{u_1} \circ {\Delta_{u_2}}^{-1} \circ \Delta_{u_3}$, $a = {\Delta_{u_1}}^{-1}(x)$, $b =(1-xu_2-yu_2)^{-1}$, 並びに, $\Delta' = {\Delta_{u_2}}^{-1} \circ \Delta_{u_3}$, $a' = x$, $b' = (1-xu_2-yu_2)^{-1}$ と元をふたつ選ぶと, 
    For the identity (iii), we take two special elements $a b \tau(\Delta(a))$ and $a' b' \tau(\Delta'(a'))$, where $\Delta = \Delta_{u_1} \circ {\Delta_{u_2}}^{-1} \circ \Delta_{u_3}$, $a = {\Delta_{u_1}}^{-1}(x)$, and $b =(1-xu_2-yu_2)^{-1}$; $\Delta' = {\Delta_{u_2}}^{-1} \circ \Delta_{u_3}$, $a' = x$, and $b' = (1-xu_2-yu_2)^{-1}$. 
    Then, we see 
    \[
          \frac{x}{1-xu_1}(1-xu_1-yu_1) \frac{1}{1- \frac{1}{1-xu_3}yu_2} \frac{1}{1-xu_3}y = a b \tau(\Delta(a))\in \D^{0}_{\Delta}\; , 
    \]
    \[
          x \frac{1}{1- \frac{1}{1-xu_3}yu_2} \frac{1}{1-xu_3}y = a' b' \tau(\Delta'(a')) \in \D^0_{\Delta'}. 
    \]
    % がわかる. 
    % 定理\ref{thm:Ker}から, それぞれの元が条件式\eqref{eqn:DD}を満たすので, 両式の差をとることで(iii)の等式が得られる. 
    By substituting each of the above two elements into Equation \eqref{eqn:DD} and subtracting these equations, the identity (iii) is exactly obtained. 
    %
    % (iv)について, $\Delta = {\Delta_{u_1}}^{-1} \circ \Delta_{u_2} \circ \Delta_{u_3}$, $a = \Delta_{u_1} \circ {\Delta_{u_2}}^{-1}(x)$, $b = (1-xu_2-yu_2)^{-1} (1-xu_3-yu_3)$, ならびに, $\Delta' = {\Delta_{u_1}}^{-1} \circ \Delta_{u_2}$, $a' = \Delta_{u_1} \circ {\Delta_{u_2}}^{-1}(x)$, $b' = (1-xu_2-yu_2)^{-1}$ と元をふたつ選ぶと, 
    Similarly, for the identity (iv), we also take two special elements $a b \tau(\Delta(a))$ and $a' b' \tau(\Delta'(a'))$, where $\Delta = {\Delta_{u_1}}^{-1} \circ \Delta_{u_2} \circ \Delta_{u_3}$, $a = \Delta_{u_1} \circ {\Delta_{u_2}}^{-1}(x)$, and $b = (1-xu_2-yu_2)^{-1} (1-xu_3-yu_3)$; $\Delta' = {\Delta_{u_1}}^{-1} \circ \Delta_{u_2}$, $a' = \Delta_{u_1} \circ {\Delta_{u_2}}^{-1}(x)$, and $b' = (1-xu_2-yu_2)^{-1}$. 
    % がわかる. 
    % 定理\ref{thm:Ker}から, それぞれの元が条件式\eqref{eqn:DD}を満たすので, 両式の差をとることで(iv)の等式が得られる. 
    By substituting each element into Equation \eqref{eqn:DD} and subtracting these equations, the identity (iv) is exactly obtained, too. 
\end{proof}
%
% 系\ref{cor:explicit}において, 各等式の両辺から不定元の係数を取り出しても等式が成り立つので, (ii)と(iv)から次の系を得る. ((i)と(iii)については系\ref{cor:in deri}の後に述べる. )
Considering the left-hand sides of identities (ii) and (iv), we newly point out that the two families of the duality relations can be generated by the derivation relation. 
%
%\newpage
%
\begin{cor} \label{cor:in deri} \mbox{}
    \begin{enumerate}
        \item[(i)] For positive integers $d$, $m$, and $n$, we have % 任意の正整数 $d, m, n$ に対して, 
        \[
            (\tau-\mathrm{id})\left(
                \sum_{\substack{m_1 + \cdots + m_d = m \\ m_1, \ldots, m_d \geq 0}} \; 
                \sum_{\substack{n_1 + \cdots + n_d = n \\ n_1, \ldots, n_d \geq 0}} 
                    xy^{m_1} x^{n_1}y \cdots xy^{m_d} x^{n_d}y
            \right)
            \in \partial(\Hoff^0). 
        \]
        \item[(ii)] For positive integers $k$, $r$, and $m$ ($k > r + m$), we have % 任意の正整数 $k$, $r$, $m$ ($k > r + m$)に対して, 
        \[
            (\tau-\mathrm{id})\left(
                \sum_{\substack{m_1 + \cdots + m_{r} = k-r-m-1 \\ m_1, \ldots, m_{r} \geq 0}}
                    x( x^{m_1}y \cdots x^{m_{r}}y) x^{m-1} y
            \right)
            \in \partial(\Hoff^0). 
        \]
    \end{enumerate}
\end{cor}
Corollary \ref{cor:in deri} (i) follows by considering the part of Corollary \ref{cor:explicit} (ii) that satisfies degree $m$ in $u_1$ and degree $n$ in $u_2$. 
Similarly, Corollary \ref{cor:in deri} (ii) also follows by considering the part of Corollary \ref{cor:explicit} (iv) that satisfies total degree $k-2$ in $u_1$, $u_2$, and $u_3$, degree $r-1$ in $u_1$, and degree $m$ in $u_3$. 
%
% 系\ref{cor:in deri} (i), (ii)に写像 $Z$ を作用させることで特殊なインデックスについての多重ゼータ値の和に関する双対関係式を得る. 
Applying the map $Z$ to Corollary \ref{cor:in deri} (i) and (ii), we find that the derivation relations can deduce these duality relations for the sum of MZVs with special indices. 
% 系\ref{cor:in deri} (i)では, 重さが $m+n+2d$ で, $(1+ \{ 1, \}^{m_1} \{ 1+ \}^{n_1} 1, \ldots, 1+ \{1, \}^{m_d} \{ 1+ \}^{n_d} 1)$で
% 表されるインデックスをもつ多重ゼータ値の和についての双対関係式が, 導分関係式から生成できることを表している. 
From Corollary \ref{cor:in deri} (i), we obtain the duality relation for the sum of MZVs with indices $(1+ \{ 1, \}^{m_1}\; n_1+ 1, \ldots, 1+ \{1, \}^{m_d}\; n_d+ 1)$, where $\{a\}^{m}$ denotes $\underbrace{a \cdots a}_m$. 
Similarly, from Corollary \ref{cor:in deri} (ii), we also obtain the duality relation for the sum of MZVs with weight $k$, depth $r+1$, and $k_{r+1} = m$. 
% 特に, $d=1$ とすると, $\zeta(2, \{1\}^{m -1}, {n +1})$ についての双対関係式; 
%すなわち, $\zeta(2, \{1\}^{m -1}, {n +1}) = \zeta(2, \{1\}^{n -1}, {m +1})$ が導分関係式から生成できることを表し, 先行研究からは導けない結果である. 
In particular, taking $d=1$ in Corollary \ref{cor:in deri} (i), we find that the simple duality relation $\zeta(2, \{1,\}^{m -1} {n +1}) - \zeta(2, \{1,\}^{n -1} {m +1})$ can be generated by the derivation relation. This result is not reported in Kajikawa \cite{Kaji}, Kawasaki and Tanaka \cite{KT}, or Li \cite{Li}, so it is newly pointed out here. 
% 系\ref{cor:in deri} (ii)では, 重さを $k$, 深さを $r+1$ と固定し, 右端の成分 $k_r$ が $m$ である収束インデックスをもつ多重ゼータ値の和に関する双対関係式が, 導分関係式から生成できることを表している. 
% 重さは $u_1, u_2, u_3$ の次数の和, 深さは $u_1$ の次数, $k_r$ は $u_3$ の次数によって, それぞれ固定される. 

% 同様に, 系\ref{cor:explicit} (i)と(iii)についても不定元の係数を取り出せる. 
Of course, similar results also follow from each of Corollary \ref{cor:explicit} (i) and (iii). 
% 系\ref{cor:explicit} (i)の両辺から各係数を取り出すと, 左辺は重さ・深さ・高さを固定した多重ゼータ値の和についての双対関係式に対応しており, 右辺の導分関係式から生成できることを表している. 
The left-hand side of Corollary \ref{cor:explicit} (i) coincides with the duality relation for the sum of MZVs with fixed weight, depth, and height. 
%Hence this duality relation can be generated by the derivation relation of the right side, explicitly. 
% このとき, 重さは $u_1$ と $u_2$ の次数の和,  深さは $u_2$ の次数, 高さは $d$ によって, それぞれ固定される. 
Here, the weight, depth, and height are fixed by the total degree in $u_1$ and $u_2$, the degree in $u_2$, and $d$, respectively. 
% 系\ref{cor:explicit} (iii)の両辺から各係数を取り出すと, 重さ・深さ・インデックスの左端 $k_1$ を固定した多重ゼータ値の和についての双対関係式に対応しており, 右辺の導分関係式から生成できることを表している. 
Similarly, the left-hand side of Corollary \ref{cor:explicit} (iii) expresses the duality relation for the sum of MZVs with fixed weight, depth, and the first component $k_1$ of the index. %; so this duality relation can also be generated by the derivation relation. 
% 重さは $u_1, u_2, u_3$ の次数の和, 深さは $u_2$ の次数, $k_1$ は $u_1$ の次数によってそれぞれ固定される. 
Here, the weight, depth, and $k_1$ are also fixed by the total degree in $u_1$, $u_2$, and $u_3$, the degree in $u_2$, and the degree in $u_3$, respectively. 
% 実際, 系\ref{cor:explicit} (i)は\cite{Kaji}*{Main Theorem}の等式と同値である. （詳しくは, Appendix 命題\ref{prop:equivalent}を参照されたい. ）
In fact, Corollary \ref{cor:explicit} (i) is equivalent to the identity of Kajikawa \cite[Main Theorem]{Kaji} (Proposition \ref{prop:equivalent}), 
% 実際, 系\ref{cor:explicit} (iii)は\cite{Li}*{Theorem 1}の等式と同値である. 
and Corollary \ref{cor:explicit} (iii) is equivalent to the identity of Li \cite[Theorem 1]{Li} (Proposition \ref{prop:equivalent2}). 
The proofs are given in Appendix A.  
% 加えて, \cite{Li}*{Theorem 1}と同様に, 系\ref{cor:explicit} (iii)において $u_1=1$, $u_2=0$ とすると, その式は, 定理\cite{KT}*{Theorem (i)}の式を不定元 $u_3$ で次数について母関数化したものである. 
Note that Corollary \ref{cor:explicit} (iii) contains both results given by Kawasaki and Tanaka \cite[Theorem (i) and (ii)]{KT} as well as the result of Li \cite[Theorem 1]{Li}. We mention this in Remark \ref{rem:KT}. 
%
%%%%%%%%%%%%%%%%%%%%%%%%%%%%%%%%%%%%%
\section{Proof of Theorems} \label{sec:proof}
%%%%%%%%%%%%%%%%%%%%%%%%%%%%%%%%%%%%%
%
%%%%%%%%%%%%%%%%%%%
%
% 定理\ref{thm:Ker}を示したのち, 定理\ref{thm:dual cap deri}の証明を与える. 
In this section, we prove Theorems \ref{thm:Ker} and \ref{thm:dual cap deri}. 
First, to prove Theorem \ref{thm:Ker}, we show the key lemma as follows. 
\begin{lem}[] \label{lem:tD}
    % 非可換べき級数環 $\Hoff\llbracket u_1, \ldots, u_s\rrbracket $ 上で, 次が成り立つ: 
    On $\Hoff\llbracket u_1, \ldots, u_s\rrbracket $, we have 
    \[
        \tau \circ {\Delta_{u_i}}^{-1} \circ \tau \equiv \Delta_{u_i} \quad (1 \leq i \leq s). 
    \]
\end{lem}
\begin{proof}
    % 主張を示すには, $\Hoff\llbracket u_1, \ldots, u_s\rrbracket $ 上で
    % $\tau \circ {\Delta_{u_i}}^{-1} \circ \tau$ が $\Hoff\llbracket u_1, \ldots, u_s\rrbracket $の環準同型写像であり, 写像 $\tau \circ \Delta_{u_i} \circ \tau$ と $\Delta_{u_i}$ による生成元 $x, y, u_1, \ldots, u_s$ の像がそれぞれ等しいことを確認すればよい．
    It is easy to check that the map $\tau \circ {\Delta_{u_i}}^{-1} \circ \tau$ is a ring homomorphism. For each generator $x$, $y$, and $u_j$ ($1 \leq j \leq s$), we see that
    \[
        \tau \circ {\Delta_{u_i}}^{-1} \circ \tau(x) = \Delta_{u_i}(x), \quad \tau \circ {\Delta_{u_i}}^{-1} \circ \tau(y) = \Delta_{u_i}(y), \quad \tau \circ {\Delta_{u_i}}^{-1} \circ \tau(u_j) = \Delta_{u_i}(u_j). 
    \]
\end{proof}
By using Lemma \ref{lem:tD}, we prove Theorem \ref{thm:Ker}. 
\begin{proof}[Proof of Theorem \ref{thm:Ker}]
    % ある元 $w \in \Hoff\llbracket u_1, \ldots, u_s\rrbracket $ が条件式\eqref{eqn:DD}を満たすことと, $w$ が$ \mathrm{Ker}(\Delta - \tau)$ の元であることは明らかに同値である. 
    It is straightforward that $w \in \mathrm{Ker}(\Delta - \tau)$ if and only if $w$ satisfies Equation \eqref{eqn:DD}. 
    % よって, $\mathrm{Ker}(\Delta - \tau) = \D_{\Delta}$ を示せばよい. 
    Hence, we prove the following equation: 
    \begin{equation}
        \mathrm{Ker}(\Delta - \tau) = \D_{\Delta}. \label{eqn:Ker=D}
    \end{equation}
    % まず, $\mathrm{Ker}(\Delta - \tau) \supset \D_{\Delta}$ を示す. 
    %to check $\mathrm{Ker}(\Delta - \tau) \supset \D_{\Delta}$, 
    % ベクトル空間 $\D_{\Delta}$ の任意の生成元 $a b \tau(\Delta(a))$
    Since $\mathrm{Ker}(\Delta - \tau)$ is a $\Q$-linear space, we first check that each generator $ab\tau(\Delta(a)) \in \D_{\Delta}$ belongs to $\mathrm{Ker}(\Delta - \tau)$, where $a \in \Hoff\llbracket u_1, \ldots, u_s\rrbracket$ and $b \in (\Q[z])\llbracket u_1, \ldots, u_s\rrbracket$ ($z = x+y$). By using Lemma \ref{lem:tD} and recalling $\Delta(z) = \tau(z) =z$, we obtain the following: 
    % $a b \tau(\Delta(a))$ of $\D_{\Delta}$ 
    % ($a \in \Hoff\llbracket u_1, \ldots, u_s\rrbracket $, $b \in \Q[z]\llbracket u_1, \ldots, u_s\rrbracket $)について, 補題\ref{lem:tD}と$\Delta(b) = \tau(b) =b$であることを用いると, 
    \begin{align*}
        (\Delta - \tau)(a b \tau(\Delta(a)))
        &= \Delta \left( a b \tau(\Delta(a)) \right) - \tau \left( a b \tau(\Delta(a)) \right) \\
        &= \Delta(a) \Delta(b) \Delta(\tau(\Delta(a))) - \tau \left( \tau(\Delta(a)) \right) \tau(b) \tau(a) \\
        &= \Delta(a) \Delta(b) \Delta(\Delta^{-1}(\tau(a))) - \Delta(a) \tau(b) \tau(a) \\
        &= \Delta(a) b \tau(a) - \Delta(a) b \tau(a) \\
        &= 0. 
    \end{align*}
    % がわかるので, $\mathrm{Ker}(\Delta -\tau) \supset \D_\Delta$ を得る. 
    Thus, $\mathrm{Ker}(\Delta -\tau) \supset \D_\Delta$ is obtained. 
    % 次に $\mathrm{Ker}(\Delta - \tau) \subset \D_{\Delta}$ を示す. 
    Next, we check that $\mathrm{Ker}(\Delta - \tau) \subset \D_{\Delta}$. 
    % 任意の $w \in \mathrm{Ker}(\Delta - \tau)$ について, $w = \tau(\Delta(w))$ であるので, 
    For $w \in \mathrm{Ker}(\Delta - \tau)$, it is easy to see that $w = \tau(\Delta(w))$, so we have the following: 
    \begin{align*}
        w &= \frac{1}{2}w + \frac{1}{2}\tau(\Delta(w))\\
        &= \frac{1}{2}w\frac{1}{1-x-y}(1-x-y) + \frac{1}{2}(1-x-y)\frac{1}{1-x-y}\tau(\Delta(w)) \\
        &= \frac{1}{2} (w + 1-x-y) \frac{1}{1-x-y} \{ \tau(\Delta(w)) + 1-x-y \} \\
        & \quad -\frac{1}{2} w \frac{1}{1-x-y} \tau(\Delta(w)) -\frac{1}{2} (1-x-y) \frac{1}{1-x-y} (1-x-y) \\
        &= \frac{1}{2}(w + 1-x-y) \frac{1}{1-x-y}\tau(\Delta(w + 1-x-y)) \\
        & \quad -\frac{1}{2} w \frac{1}{1-x-y} \tau(\Delta(w)) -\frac{1}{2} (1-x-y) \frac{1}{1-x-y} \tau(\Delta(1-x-y)). 
    \end{align*}
    % を得る. 
    % 最右辺の各項は $\Q$-ベクトル空間 $\D_\Delta$ の生成元であるので, $w \in \D_{\Delta}$ である. 
    Since each of the three terms is a generator of $\D_\Delta$, we find that $w \in \D_\Delta$. 
    % したがって, $\mathrm{Ker}(\Delta - \tau) = \D_{\Delta}$ が示された. 
    In other words, $\mathrm{Ker}(\Delta -\tau) \subset \D_\Delta$, and so Equation \eqref{eqn:Ker=D} holds. 
\end{proof}
%
% 定理 \ref{thm:dual cap deri} を証明するために以下の補題を用いる．
Next, we show the following lemma needed for the proof of Theorem \ref{thm:dual cap deri}. 
\begin{lem}[] \label{lem:Ker=Im} \mbox{}
    On $\Hoff\llbracket u_1, \ldots, u_s\rrbracket $, we have the following identities: 
    % 非可換べき級数環 $\Hoff\llbracket u_1, \ldots, u_s\rrbracket $ 上の自己同型写像 $\Delta = {\Delta_{u_1}}^{e_1} \circ \cdots \circ {\Delta_{u_s}}^{e_s}$ ($e_1, \ldots, e_s$は整数)について, 以下が成り立つ: 
    \begin{enumerate}
        \item[(i)] 
        \[
            \mathrm{Ker}(\Delta - \tau) = \mathrm{Im}(\Delta^{-1} + \tau). 
        \]
        \item[(ii)] 
        \[
            \mathrm{Ker}(\Delta + \tau) = \mathrm{Im}(\Delta^{-1} - \tau). 
        \]
    \end{enumerate}
%    また, $\Hoff^0\llbracket u_1, \ldots, u_s\rrbracket $上に制限しても主張が成り立つ. 
\end{lem}
Note that we use only Lemma \ref{lem:Ker=Im} (i) in the proof of Theorem \ref{thm:dual cap deri}. However, Lemma \ref{lem:Ker=Im} (ii) is used in the proof of Proposition \ref{prop:cap}. 
\begin{proof}
    % (i) まず, 像 $\mathrm{Im}(\Delta^{-1} + \tau)$ の元 $(\Delta^{-1} + \tau)(w)$ ($w \in \Hoff\llbracket u_1, \ldots, u_s\rrbracket $)について, 写像 $\Delta - \tau$ を作用させると, 補題\ref{lem:tD}により, 
    We prove Lemma \ref{lem:Ker=Im} (i). By using Lemma \ref{lem:tD}, we obtain the following: 
    \[
        (\Delta - \tau) \left( (\Delta^{-1} + \tau)(w) \right) 
        = \left( \mathrm{id} +\Delta \circ \tau -\tau \circ \Delta^{-1} -\mathrm{id} \right)(w) 
        = 0
    \]
    for any $w \in \Hoff\llbracket u_1, \ldots, u_s \rrbracket$. 
    Thus, we have $\mathrm{Ker}(\Delta - \tau) \supset \mathrm{Im}(\Delta^{-1} + \tau)$. 
    % 次に, 任意の $w \in \mathrm{Ker}(\Delta - \tau)$ について, $w = \Delta^{-1}(\tau(w))$ であるので, 
    Next, reviewing $w = \Delta^{-1}(\tau(w))$ for any $w \in \mathrm{Ker}(\Delta - \tau)$, we have 
    \begin{align*}
        w &= \frac{1}{2} \left\{ \Delta^{-1}(\tau(w))+ w \right\} \\
        &= \frac{1}{2} \left\{ \Delta^{-1}(\tau(w)) + \tau(\tau(w)) \right\} \\
        &= (\Delta^{-1} + \tau) \left( \frac{1}{2}\tau(w) \right) \quad \in \mathrm{Im}(\Delta^{-1} + \tau). 
    \end{align*}
    %がわかり, $\mathrm{Ker}(\Delta - \tau) \subset \mathrm{Im}(\Delta^{-1} + \tau)$ である. 
    % Thus we have $\mathrm{Ker}(\Delta - \tau) \subset \mathrm{Im}(\Delta^{-1} + \tau)$. 
    % したがって, $\mathrm{Ker}(\Delta - \tau) = \mathrm{Im}(\Delta^{-1} + \tau)$. 
    Therefore, we have $\mathrm{Ker}(\Delta - \tau) \subset \mathrm{Im}(\Delta^{-1} + \tau)$, and so $\mathrm{Ker}(\Delta - \tau)$ and $\mathrm{Im}(\Delta^{-1} + \tau)$ are equal. 
    %(ii)についても同様の議論で $\mathrm{Ker}(\Delta + \tau) = \mathrm{Im}(\Delta^{-1} - \tau)$が示される. 
    The proof of Lemma \ref{lem:Ker=Im} (ii) is similar and so will be omitted. 
\end{proof}
Using these properties, we prove Theorem \ref{thm:dual cap deri}. 
\begin{proof}[Proof of Theorem \ref{thm:dual cap deri}]
    % (ii)と(iii)の集合が等しいことは, 定理\ref{thm:Ker}から容易にわかる. さらに, 式\eqref{eqn:DD}からどちらの集合も(i)の集合 $(\tau - \mathrm{id})(\Hoff) \cap \partial (\Hoff)$ に包含される. よって, (i)の集合が(ii)の集合に包含されることを確認すればよい. 
    From Theorem \ref{thm:Ker}, it is easy to see that the sets (ii) and (iii) are equal and that both of these sets are contained in the set (i), i.e., (i) $\supset$ (ii) $=$ (iii). 
    To prove (i) $\subset$ (ii), it suffices to check the following: 
    % ここで, $\theta(\Hoff)$ を, 写像 $\theta_l$ による $\Hoff$ の像によって張られるベクトル空間とすると, 式\eqref{eqn:derivation}から $\partial(\Hoff)=\theta(\Hoff)$ がわかるので, (i)と(ii)の集合について, 以下の包含関係を確認すればよい: 
    \begin{gather}
        (\tau - \mathrm{id})(\Hoff) \cap \partial (\Hoff) \;\subset\; (\tau -\mathrm{id})(\partial(\Hoff)), \label{inclusion 1} \\
        (\tau -\mathrm{id})(\partial(\Hoff)) \;\subset\; \mathrm{Coef}_{u_1, \ldots, u_s} \left( (\tau -\mathrm{id})(\D_{\Delta}) \right). \label{inclusion 2}
    \end{gather}
    % まず, 集合 $(\tau - \mathrm{id})(\Hoff) \cap \partial (\Hoff)$ の任意の元 $w$ について, $\Hoff$のある元 $w_1, \ldots, w_m$ が存在して，
    First, we check Equation \eqref{inclusion 1}. % the inclusion relation given by Equation
    Each element of $(\tau - \mathrm{id})(\Hoff) \cap \partial (\Hoff)$ can be expressed as $(\tau -\mathrm{id})(w)$, where $w \in \Hoff$. 
    Then, there exists a certain positive integer $N$ such that $N$ satisfies 
    \[
        (\tau - \mathrm{id})(w) = \sum_{n=1}^{N} \partial_{n}(v_n), 
    \]
    % ここで, $n_1 \ldots, n_m$ は正整数である. 
    where $v_n \in \Hoff$. 
    %両辺に $(\tau -\mathrm{id})$ を作用させると, 
    Applying the map $\tau -\mathrm{id}$ to both sides, we have
    \begin{align*}
        \sum_{n=1}^{N} (\tau -\mathrm{id}) \left( \partial_{n}(v_n) \right) 
        &= (\tau -\mathrm{id})^2(w) \\
        &= -2(\tau -\mathrm{id})(w). 
    \end{align*}
    % よって, $(\tau -\mathrm{id})(w) \in (\tau -\mathrm{id}) \left( \partial(\Hoff) \right)$ が成り立ち, 包含関係\eqref{inclusion 1}が示された. 
    Recalling that $(\tau -\mathrm{id}) \left( \partial(\Hoff) \right)$ is a $\Q$-linear space, we obtain $(\tau -\mathrm{id})(w) \in (\tau -\mathrm{id}) \left( \partial(\Hoff) \right)$. Therefore, Equation \eqref{inclusion 1} holds. 
    %
    % 次に，集合 $\mathrm{Coef}_{u_1, \ldots, u_s} \left((\tau -\mathrm{id})(\D_{\Delta} \right)$ について, 定義から $\Q$-ベクトル空間であることが容易に確かめられるので, 任意の $\theta_n$, $w$ に対して, 
    Next, we prove Equation \eqref{inclusion 2}. 
    Since the set $(\tau -\mathrm{id})(\D_{\Delta})$ is a $\Q$-linear space, the set $\mathrm{Coef}_{u_1, \ldots, u_s} \left((\tau -\mathrm{id})(\D_{\Delta}) \right)$ is a $\Q$-linear space, too. 
    Reviewing that all $\theta_{n}(w)$ ($n \in \Z$, $w \in \Hoff$) are generators of $\partial(\Hoff)$ from Theorem \ref{thm:derivation}, it suffices to check that
    \begin{equation}
        (\tau -\mathrm{id}) \left( \theta_n(w) \right) \in \mathrm{Coef}_{u_1, \ldots, u_s} \left( (\tau -\mathrm{id})(\D_{\Delta}) \right) \label{eqn:generator}
    \end{equation}
    for any integer $n$ and $w \in \Hoff$. 
    % を示せば, 包含関係が確認される. 
    % ここで, 各不定元 $u_1, \ldots, u_s$ は対称的なので, $e_1 \neq 0$ を仮定しても一般性は失われない. 
    We may assume $e_1 \neq 0$ without loss of generality. 
    % この仮定を採用して, $\Delta = {\Delta_{u_1}}^{e_1}$ ($e_1 > 0$) であるとき, 任意の$w \in \Hoff$ で\eqref{eqn:generator}が成り立つことを正整数 $n$ についての帰納法で示す. 
    First, in the case of $\Delta = {\Delta_{u_1}}^{e_1}$ ($e_1 > 0$), we prove Equation \eqref{eqn:generator} by induction on $n$. We will take the element such that one of its coefficients is $(\tau -\mathrm{id}) \left( \theta_n(w) \right)$. %the membership relation given by 
    % 定理\ref{thm:Ker}の証明と補題$\ref{lem:Ker=Im}$(i)から, $\D_{\Delta} = \mathrm{Ker}(\Delta -\tau) = \mathrm{Im}(\Delta^{-1} + \tau)$ が成り立つので, 
    By using Equation \eqref{eqn:Ker=D} and Lemma \ref{lem:Ker=Im} (i), we obtain
    % $\D_{\Delta} = \mathrm{Ker}(\Delta -\tau) = \mathrm{Im}(\Delta^{-1} + \tau)$. 
    \begin{equation}
        \D_{\Delta}
        = \mathrm{Ker}(\Delta - \tau) = \mathrm{Im}(\Delta^{-1} + \tau). \label{eqn:D=Im}
    \end{equation}
    % がわかる. 
    %
    % 集合 $(\tau -\mathrm{id})(\D_{\Delta_{u_1}^{e_1}})$ の元として, $(\tau-\mathrm{id}) \left( ({\Delta_{u_1}}^{-e_1} + \tau)(\tau(w)) \right)$ を選び, 補題$\ref{lem:tD}$を適用すると, 
    Thus, for any element $w \in \Hoff$, we have
    \[
        ({\Delta_{u_1}}^{-e_1} + \tau)(w) \in \D_{\Delta_{u_1}^{e_1}}. 
    \]
    Since the map $\tau$ is an anti-automorphism of $\Hoff$, let us replace $w$ with $\tau(w)$. 
    By using Lemma $\ref{lem:tD}$, we obtain
    \[
        ({\Delta_{u_1}}^{-e_1} + \tau)(\tau(w)) = (\tau \circ {\Delta_{u_1}}^{e_1} + \tau \circ \tau)(w) = -({\Delta_{u_1}}^{e_1} + \tau)(w). 
    \]
    The point is that $(\tau - \mathrm{id})(({\Delta_{u_1}}^{e_1} + \tau)(w))$ is exactly the element we wanted. 
    Now, in the expansion of ${\Delta_{u_1}}^{e_1}$, we consider the coefficients of ${u_1}^i$: 
    \begin{align*}
        {\Delta_{u_1}}^{e_1} &= \Biggl( \sum_{i \geq 0} \theta_i {u_1}^i \Biggr)^{e_1} \\
        &= \mathrm{id} + \binom{e_1}{1} \theta_1 u_1 + \left\{ \binom{e_1}{2} {\theta_1}^2 + \binom{e_1}{1}\theta_2 \right\} {u_1}^2 + \cdots \\
        &= \sum_{i \geq 0} \sum_{\substack{j_0 + j_1 + \cdots + j_i = e_1\\ j_1 + 2 \cdot j_2 \cdots + i \cdot j_i = i\\ j_0, j_1, \ldots, j_i \geq 0}} \; \frac{e_1!}{j_0! j_1! \cdots j_i!} {\mathrm{id}}^{j_0} \circ {\theta_1}^{j_1} \circ \cdots \circ {\theta_i}^{j_i} {u_1}^{i}, 
    \end{align*}
    % と変形できる. 
    %ゆえに, $-(\tau - \mathrm{id})\left( ( {\Delta_{u_1}}^{e_1} + \tau )(w) \right)$における${u_1}$の係数は
    %$-e_1 (\tau -\mathrm{id})( \theta_1(w) )$ であるので, $(\tau -\mathrm{id}) \left( \theta_1(w) \right) \in \mathrm{Coef}_{u_1} \left( (\tau -\mathrm{id})(\D_{\Delta_{u_1}^{e_1}}) \right)$; 
    where $\binom{e}{j}$ is the binomial coefficient. From the above, we find that the coefficient of $u_1$ is just $e_1 (\tau -\mathrm{id})( \theta_1(w) )$, and so the set membership holds for $n=1$ in the case of $\Delta = {\Delta_{u_1}}^{e_1}$. 
    %
    % 次に, 帰納法の仮定として, $n-1$ 以下の任意の正整数 $i$ に対して，
    % $(\tau -\mathrm{id})( \theta_i(w) )$ は $\mathrm{Coef}_{u_1} \left( (\tau -\mathrm{id})( \D_{\Delta_{u_1}^{e_1}} ) \right)$ の元であるとする. 
    Suppose $n \geq 2$ and assume the following set membership holds for any $i < n$ in the case of $\Delta = {\Delta_{u_1}}^{e_1}$ ($e_1 > 0$): 
    \[
        (\tau - \mathrm{id})( \theta_i(w) ) \in \mathrm{Coef}_{u_1} ( (\tau -\mathrm{id})( \D_{\Delta_{u_1}^{e_1}} ) ). 
    \]
    % 先の議論と同様に, $(\tau -\mathrm{id})\left( ( {\Delta_{u_1}}^{e_1} + \tau )(w) \right)$ の ${u_1}^{n}$ の係数は
    % $\mathrm{Coef}_{u_1} \left( (\tau -\mathrm{id})( \D_{\Delta_{u_1}^{e_1}} ) \right)$ の元であり, その係数は
    Similarly, for an element $(\tau -\mathrm{id})\left( ( {\Delta_{u_1}}^{e_1} + \tau)(w) \right)$, the coefficient of ${u_1}^n$ is 
    \begin{equation*}
        e_1(\tau -\mathrm{id})( \theta_n(w) ) 
        + \sum_{\substack{j_0 + j_1 + \cdots + j_{n-1} = e_1\\ j_1 + 2 \cdot j_2 + \cdots + (n-1) \cdot j_{n-1} = n\\ j_0, j_1, \ldots, j_{n-1} \geq 0}}
            \frac{e_1!}{j_0! j_1! \cdots j_{n-1}!} (\tau -\mathrm{id}) \left( ({\theta_1}^{j_1} \circ \cdots \circ {\theta_{n-1}}^{j_{n-1}}) (w) \right). 
    \end{equation*}
    %と変形できる. 
    % 帰納法の仮定より, 右辺の第2項の有限和も $\mathrm{Coef}_{u_1} \left( (\tau -\mathrm{id})( \D_{\Delta_{u_1}^{e_1}} ) \right)$ の元なので, $(\tau -\mathrm{id})( \theta_n(w) ) \in \mathrm{Coef}_{u_1} \left( (\tau -\mathrm{id})( \D_{\Delta_{u_1}^{e_1}} ) \right)$ である. 
    By the induction hypothesis, we see that the second term belongs to $\mathrm{Coef}_{u_1}(( \tau -\mathrm{id})( \D_{\Delta_{u_1}^{e_1}} ))$. Since $\mathrm{Coef}_{u_1}(( \tau -\mathrm{id})( \D_{\Delta_{u_1}^{e_1}} ))$ is a $\Q$-linear space, $e_1(\tau - \mathrm{id})(\theta_n(w)) \in \mathrm{Coef}_{u_1}(( \tau -\mathrm{id})( \D_{\Delta_{u_1}^{e_1}} ))$. 
    That is, Equation \eqref{eqn:generator} is valid in the case of $\Delta = {\Delta_{u_1}}^{e_1}$ ($e_1 > 0$). 
    % したがって, $\Delta = {\Delta_{u_1}}^{e_1}$ ($e_1 > 0$)のとき, 任意の正整数 $n$と, $\Hoff$ の元 $w$ について, \eqref{eqn:generator}が成り立つ. 
    %
    % 一方, $e_1$ が負の整数の場合について, 式\eqref{eqn:thmlem} から, $(\tau -\mathrm{id})({\Delta_{u_1}}^{-e_1} + \tau)(w)$ は $\mathrm{Coef}_{u_1} \left( (\tau -\mathrm{id})( \D_{\Delta_{u_1}^{e_1}} ) \right)$ の元であり, 正整数 $-e_1$ の場合とみなせて\eqref{eqn:generator}が成り立つ. 
    Second, in the case of $\Delta = {\Delta_{u_1}}^{e_1}$ ($e_1 < 0$), we can apply the same argument to coefficients of $(\tau -\mathrm{id})({\Delta_{u_1}}^{-e_1} + \tau)(w)$ taken from $(\tau -\mathrm{id})(\D_{\Delta_{u_1}^{e_1}})$. 
    %    以上から, $\Delta = {\Delta_{u_1}}^{e_1}$ ($e_1 \neq 0$) の場合に\eqref{eqn:generator}が成り立つ. 
    Since $-e_1$ is also a positive integer, Equation \eqref{eqn:generator} is valid in the case of $\Delta = {\Delta_{u_1}}^{e_1}$ ($e_1 \neq 0$). 
    % 最後に, $e_1 \neq 0$ である $\Delta = {\Delta_{u_1}}^{e_1} \circ \cdots \circ {\Delta_{u_s}}^{e_s}$ の場合について, 
    %式 $\eqref{eqn:thmlem}$ において, $\Hoff$ の元として ${\Delta_{u_2}}^{e_2} \circ \cdots \circ {\Delta_{u_s}}^{e_s}(\tau(w))$ を選ぶと, 
    Finally, in the case of $\Delta = {\Delta_{u_1}}^{e_1} \circ \cdots \circ {\Delta_{u_s}}^{e_s}$ ($e_1, \ldots, e_s \in \mathbb{Z}$, $e_1 > 0$), let $v=(\Delta^{-1} + \tau)({\Delta_{u_2}}^{e_2} \circ \cdots \circ {\Delta_{u_s}}^{e_s}(\tau(w)))$. Then, the element $v$ belongs to $\D_{\Delta}$ because of Equation \eqref{eqn:D=Im}, and the element $v$ is easily transformed into ${\Delta_{u_1}}^{-e_1}(\tau(w)) + \tau\left( {\Delta_{u_2}}^{e_2} \circ \cdots \circ {\Delta_{u_s}}^{e_s}(\tau(w)) \right)$. 
    % を得る. 
    % 補題\ref{lem:tD}を用いると, $(\tau-\mathrm{id})\left( {\Delta_{u_1}}^{-e_1}(\tau(w) \right) = -(\tau-\mathrm{id})({\Delta_{u_1}}^{e_1}(w))$ であるので, 
    % 上の元における ${u_1}^i$ の係数を調べることで, $\Delta = {\Delta_{u_1}}^{e_1}$ の場合と同様に, \eqref{eqn:generator}が成り立つ. 
    % したがって, $e_1 \neq 0$ である任意の$\Delta$について\eqref{eqn:generator}が成り立つので, 包含関係\eqref{inclusion 2}が示された. 
    Now, we consider the coefficient of ${u_1}^i$ in the element $(\tau - \mathrm{id})(v)$. By using Lemma \ref{lem:tD}, we obtain 
    \[
        (\tau-\mathrm{id})\left( {\Delta_{u_1}}^{-e_1}(\tau(w)) \right) = (\tau-\mathrm{id})(\tau({\Delta_{u_1}}^{e_1}(w))) = -(\tau-\mathrm{id})({\Delta_{u_1}}^{e_1}(w)). 
    \]
    That is, the first term of $(\tau - \mathrm{id})(v)$ is $-(\tau-\mathrm{id})({\Delta_{u_1}}^{e_1}(w))$. 
    Since the second term of $v$ has no $u_1$, it suffices to consider the coefficient of ${u_1}^i$ in $-(\tau-\mathrm{id})({\Delta_{u_1}}^{e_1}(w))$. 
    By the same argument, Equation \eqref{eqn:generator} is valid in this case. 
    Similarly, in the case of $e_1 < 0$, this follows by replacing $\tau(w)$ with $w$. 
    Therefore, Equation \eqref{eqn:generator} holds, and Equation \eqref{inclusion 2} is obtained. 
\end{proof}
\appendix
\def\thesection{A}
%%%%%%%%%%%%%%%%%%%%%%%%%%%%
% Appendix
%%%%%%%%%%%%%%%%%%%%%%%%%%%%

\section{Equivalence of Previous Works and Corollary \ref{cor:explicit}} \label{sec:equivalence}
%
% 本節では, 系$\ref{cor:explicit}$(i)と\cite[Main Theorem]{Kaji}は同値であること, 
% ならびに, 写像$\Delta -\tau$ と集合 $\D_\Delta$ の性質を紹介する. 
%% 不定元 $x$, $y$ で生成される非可換形式的べき級数環 $\Hhat := \Q \langle \langle x, y \rangle \rangle$ 上で次が成り立つ. 
In this appendix, we show that Corollary \ref{cor:explicit} (i) is equivalent to the result of Kajikawa \cite{Kaji} and that Corollary \ref{cor:explicit} (iii) is equivalent to the result of Li \cite{Li}. We also mention the result obtained by Kawasaki and Tanaka \cite{KT} in Remark \ref{rem:KT}. 
To state the result of Kajikawa, let $\Hhat$ be the formal power series ring $\Q \llbracket x, y \rrbracket$ generated by the indeterminates $x$ and $y$ over $\Q$. 
\begin{thm}[Kajikawa \cite{Kaji}] \label{thm:Kaji}
    For positive integers $r$ and $d$ with $r \geq d$, we have the following identity on $\Hhat$: 
    \begin{align*}
        (\tau - \mathrm{id}) \! \left(
            \sum_{\substack{i_1 + \cdots + i_d = r \\ i_1, \ldots, i_d \geq 1}} \; 
                \prod_{j=1}^{d} \frac{x}{1-x}y^{i_j}
        \right) \!
        &= \sum_{m \geq 0} \theta_m \! \left(
            \sum_{\substack{i_1 + \cdots + i_d = r \\ i_1, \ldots, i_d \geq 1}} \; 
                \prod_{j=1}^{d}\left\{
                    \Bigl( x-\frac{x}{1-x}y \Bigr)^{i_j -1}\frac{x}{1-x}y
                \right\}
        \right) \! \\
        &-\sum_{n=0}^{r - d} \theta_n \left(
            \sum_{\substack{i_1 + \cdots + i_d = r-n\\ i_1, \ldots, i_d \geq 1}} \; 
                \prod_{j=1}^{d}\left\{
                    \Bigl( x-\frac{x}{1-x}y \Bigr)^{i_j -1}\frac{x}{1-x}y
                \right\}
        \right), 
    \end{align*}
% ここで, $\prod^{d}_{i=1} f_i$ は $f_1 \cdots f_d$ ($f_i \in \Hhat$) として順序付きの積を表す. 
where $\prod^{d}_{j=1} w_j = w_1 \cdots w_d$. 
\end{thm}
\begin{prop} \label{prop:equivalent}
    Corollary \ref{cor:explicit} (i) and Theorem \ref{thm:Kaji} are equivalent. 
\end{prop}
\begin{proof}
    % 系\ref{cor:explicit} (i)が定理\ref{thm:Kaji}の式を母関数化したものであることを示す. 
    We prove that Corollary \ref{cor:explicit} (i) is exactly the generating function of Theorem \ref{thm:Kaji} with respect to degree. 
    % 系\ref{cor:explicit} (i)の左辺は以下のように変形される: 
    First, for the left-hand side of Corollary \ref{cor:explicit} (i), we have  
    \begin{align}
        (\tau - \mathrm{id}) \biggl( \frac{x}{1-xu_1} \frac{y}{1-yu_2} \biggr)^{d} %\notag\\ 
        &= (\tau - \mathrm{id}) \left( \sum_{r_1 \geq 0} \frac{x}{1-xu_1} y(yu_2)^{r_1}
            \cdots \sum_{r_d \geq 0} \frac{x}{1-xu_1} y(yu_2)^{r_d} \right) \notag \\
        &= \sum_{r \geq 0} (\tau - \mathrm{id}) \Biggl( \sum_{\substack{r_1 + \cdots + r_d = r\\ r_1, \ldots, r_d \geq 0}}
            \prod_{j=1}^{d} \frac{x}{1-xu_1} y^{r_j +1}
        \Biggr) {u_2}^{r} \notag \\
        &= \sum_{r \geq 0} (\tau - \mathrm{id}) \Biggl( \sum_{\substack{r_1 + \cdots + r_d = r+d\\ r_1, \ldots, r_d \geq 1}}
            \; \prod_{j=1}^{d} \frac{x}{1-xu_1} y^{r_j} \Biggr)
        {u_2}^{r} \notag \\
        &= \sum_{r \geq d} (\tau - \mathrm{id}) \Biggl( \sum_{\substack{r_1 + \cdots + r_d = r\\ r_1, \ldots, r_d \geq 1}}
            \; \prod_{j=1}^{d} \frac{x}{1-xu_1} y^{r_j} \Biggr)
        {u_2}^{r-d}. \label{eqn:Kaji_first}
    \end{align}
    % 次に系\ref{cor:explicit} (i)の右辺の写像 $\Delta_{u_1} \circ {\Delta_{u_2}}^{-1}$ の像は以下のように変形される; 
    Second, the part of $\Delta_{u_1} \circ {\Delta_{u_2}}^{-1}$ on the right-hand side is transformed as follows: 
    \begin{align}
        &\Delta_{u_1} \circ {\Delta_{u_2}}^{-1} \left( \Bigl(\frac{x}{1-xu_1} \frac{y}{1-yu_2} \Bigr)^{d} \right) \notag \\
        &= \Bigl( \frac{x}{1-xu_2} \frac{y}{1-yu_1} \Bigr)^{d} \notag \\
        &= \sum_{r \geq 0} \; \sum_{\substack{i_1 + \cdots + i_d = r\\ i_1, \ldots, i_d \geq 0}} \; 
            \prod_{j=1}^{d} \left(x^{i_j +1} \frac{y}{1-yu_1} \right)
        {u_2}^r \notag \\
        &= \sum_{r \geq 0} \; 
            \sum_{\substack{i_1 + \cdots + i_d = r \\ i_1, \ldots, i_d \geq 0}} \;
                \prod_{j=1}^{d} \Delta_{u_1} \left( \Bigl( {\Delta_{u_1}}^{-1}(x) \Bigr)^{i_j} x {\Delta_{u_1}}^{-1}(y) \right) {u_2}^r \notag \\
        &= \sum_{r \geq d}
            \Delta_{u_1} \Biggl(
                \sum_{\substack{i_1 + \cdots  + i_d = r \\ i_1, \ldots, i_d \geq 1}} \; 
                    \prod_{j=1}^{d} \left\{ \Bigl( x -\frac{x}{1-xu_1}yu_1 \Bigr)^{i_j -1} \frac{x}{1-xu_1}y \right\}
        \Biggr){u_2}^{r-d}. \label{eqn:Kaji_second}
    \end{align}
    % 最後に, 系$\ref{cor:explicit}$(i)の右辺の写像 $\mathrm{id}$ の像について, $\Delta_{u_2} = \sum_{n \geq 0} \theta_n {u_2}^{n}$ に注意すると; 
    Finally, reviewing $\Delta_{u_2} = \sum_{n \geq 0} \theta_n {u_2}^{n}$, we transform the identity map part on the right-hand side: 
    \begin{align}
        &\biggl( \frac{x}{1-xu_1} \frac{y}{1-yu_2} \biggr)^{d} \notag \\
        &=\Delta_{u_2} \circ {\Delta_{u_1}}^{-1} \left( \Bigl( \frac{x}{1-xu_2} \frac{y}{1-yu_1} \Bigr)^{d} \right) \notag \\
        &=\Delta_{u_2} \left( \Bigl(
            \frac{1}{1-{\Delta_{u_1}}^{-1}(x)u_2} x {\Delta_{u_1}}^{-1}(y)
        \Bigr)^{d} \right) \notag \\
        &= \Delta_{u_2} \Biggl( \sum_{r \geq 0} \sum_{\substack{i_1 + \cdots i_d = r \\ i_1, \ldots, i_d \geq 0}} \; 
            \prod_{j=1}^{d} \left\{ \left( {\Delta_{u_1}}^{-1}(x) \right)^{i_j} x {\Delta_{u_1}}^{-1}(y) \right\}{u_2}^{r}
        \Biggr) \notag \\
        &= \sum_{r \geq 0} \sum_{n \geq 0} \theta_n \Biggl( \sum_{\substack{i_1 + \cdots i_d = r \\ i_1, \ldots, i_d \geq 0}} \; 
            \prod_{j=1}^{d} \left\{ {\Delta_{u_1}}^{-1}(x^{i_j}) x {\Delta_{u_1}}^{-1}(y) \right\}
        \Biggr) {u_2}^{r+n} \notag \\
        &=\sum_{r \geq d}\sum_{n = 0}^{r-d} \theta_{n} \Biggl( \sum_{\substack{i_1 + \cdots i_d= r-n\\ i_1, \ldots, i_d \geq 1}} \; 
            \prod_{j=1}^{d} \left\{ {\Delta_{u_1}}^{-1}(x^{i_j -1}) x {\Delta_{u_1}}^{-1}(y) 
            \right\}
        \Biggr){u_2}^{r-d} \notag \\
        &= \sum_{r \geq d} \sum_{n = 0}^{r-d} \theta_{n} \Biggl(
                \sum_{\substack{i_1 + \cdots i_d= r-n\\ i_1, \ldots, i_d \geq 1}} \; 
                \prod_{j=1}^{d} \left\{ \left(x- \frac{x}{1-xu_1}yu_1 \right)^{i_j -1} \frac{x}{1-xu_1}y \right\}
        \Biggr){u_2}^{r-d}. \label{eqn:Kaji_third}
    \end{align}
    % 式\eqref{eqn:Kaji_first}, \eqref{eqn:Kaji_second}, \eqref{eqn:Kaji_third}から, $u_2$ に関する母関数の等式が得られ, ${u_2}^{r-d}$ の係数を取り出すと,  
    Substituting \eqref{eqn:Kaji_first}, \eqref{eqn:Kaji_second}, and \eqref{eqn:Kaji_third} into Corollary \ref{cor:explicit} (i), we obtain the identity composed of these generating functions. Then, we focus on the coefficient of ${u_2}^{r-d}$:
    \begin{multline*}
        (\tau - \mathrm{id})\Biggl(
            \sum_{\substack{i_1 + \cdots + i_d = r \\ i_1, \ldots, i_d \geq 1}} \; 
                \prod_{j=1}^{d} \frac{x}{1-xu_1}y^{i_j}
        \Biggr) \\
        = \sum_{m \geq 0} \theta_m {u_1}^m \Biggl(
            \sum_{\substack{i_1 + \cdots + i_d = r \\ i_1, \ldots, i_d \geq 1}} \; 
                \prod_{j=1}^{d}\left\{
                    \Bigl( x-\frac{x}{1-xu_1}yu_1 \Bigr)^{i_j -1}\frac{x}{1-xu_1}y
                \right\}
        \Biggr) \\
        -\sum_{n=0}^{r - d} \theta_n \Biggl(
            \sum_{\substack{i_1 + \cdots + i_d = r-n\\ i_1, \ldots, i_d \geq 1}} \; 
                \prod_{j=1}^{d}\left\{
                    \Bigl( x-\frac{x}{1-xu_1}yu_1 \Bigr)^{i_j -1}\frac{x}{1-xu_1}y
                \right\}
        \Biggr). 
    \end{multline*}
    % がわかる. 
    % この式は, 不定元 ${u_1}$ に関して, 定理\ref{thm:Kaji}の式を次数ごとに母関数化したものなので, 系\ref{cor:explicit} (i)と定理\ref{thm:Kaji}は同値である. 
    This identity is exactly the generating function of Theorem \ref{thm:Kaji} with respect to degree. 
\end{proof}
\begin{thm}[Li \cite{Li}] \label{thm:Li}
    On $\Hoff^{0}[[u_1, u_2, u_3]]$, we have
    \begin{align*}
        &(\tau - \mathrm{id})\left(\frac{x}{1-xu_1}y \frac{1}{1-xu_3-yu_2}y\right) \\
        &= -\frac{1}{u_2-u_3}(\Delta_{u_2} - \Delta_{u_3}) \biggl(
            x\frac{1}{1-xu_1-xu_2+(x^2+yx)u_1u_2}y \frac{1}{1-xu_3}(1-xu_3-yu_3)
        \biggr)\\
        &\quad -(1 - \Delta_{u_1}) \biggl(
            x\frac{1}{1-xu_1-xu_2+(x^2+yx)u_1 u_2-yu_3} (1-xu_1-yu_1) \frac{x}{1-xu_1}y
        \biggr). 
    \end{align*}
\end{thm}
\begin{prop} \label{prop:equivalent2}
    Corollary \ref{cor:explicit} (iii) and Theorem \ref{thm:Li} are equivalent. 
\end{prop}
\begin{proof}
    We check that Corollary \ref{cor:explicit} (iii) is exactly Theorem \ref{thm:Li} multiplied by $u_1$ on both sides. 
    To prove this proposition, we show the following: 
    \begin{equation}
        \frac{1}{1- \frac{1}{1-xu_3}yu_2} \frac{1}{1-xu_3} = \frac{1}{1 -xu_3 -yu_2}.  \label{eqn:XY}
    \end{equation}
    %we rewrite the left-hand side as
    %\begin{align*}
    %    \frac{1}{1- \frac{1}{1-X}Y} \frac{1}{1-X} 
    %    &= \sum_{l \geq 0} \left( \frac{1}{1-X}Y \right)^l \sum_{m \geq 0}X^m \\
    %    &= \sum_{l \geq 0} \prod_{i=1}^{l} \left( \sum_{n_i \geq 0}X^{n_i} Y \right) \sum_{m \geq 0}X^m, 
    %\end{align*}
    %where the empty product is $1$. 
    It is easy to see that $(1-xu_3)\left(1- \frac{1}{1-xu_3}yu_2 \right) = 1-xu_3-yu_2$. 
    Thus, we obtain
    \begin{align*}
        \frac{1}{1-xu_3-yu_2} &= \frac{1}{1-\frac{1}{1-xu_3}yu_2}\frac{1}{1-xu_3} (1-xu_3)\left(1- \frac{1}{1-xu_3}yu_2 \right) \frac{1}{1-xu_3-yu_2} \\
        &= \frac{1}{1-\frac{1}{1-xu_3}yu_2}\frac{1}{1-xu_3} (1-xu_3-yu_2) \frac{1}{1-xu_3-yu_2} \\
        &= \frac{1}{1-\frac{1}{1-xu_3}yu_2}\frac{1}{1-xu_3}. 
    \end{align*}
    % In the above, any word generated by $X$ and $Y$ appears only once without duplication. For example, the word $XYXY$ appears only once. 
    %Similarly, in $\frac{1}{1-X-Y}$, any word generated by $X$ and $Y$ appears only once, too. 
    %Therefore, Equation \eqref{eqn:XY} holds. 
    % By taking $X=xu_3$ and $Y=yu_2$, 
    Therefore, the left-hand side of Corollary \ref{cor:explicit} (iii) is transformed as follows: 
    \begin{equation*}
        (\tau - \mathrm{id}) \left( \frac{xu_1}{1-xu_1} y \frac{1}{1- \frac{1}{1-xu_3}yu_2} \frac{1}{1-xu_3}y \right) 
        = (\tau - \mathrm{id}) \left( \frac{x}{1-xu_1} y \frac{1}{1 - xu_3 - yu_2}y \right) u_1. 
    \end{equation*}
    The above is exactly the left-hand side of Theorem A.3 multiplied by u1. 
\end{proof}
\begin{rem}\upshape \label{rem:KT}
Similar to the method mentioned in Li \cite{Li}, formally multiplying Corollary \ref{cor:explicit} (iii) by $1/u_1$ and setting $u_1=0$ (or $u_2=0$), we obtain the generating function of the result given by Kawasaki and Tanaka \cite[Theorem (ii) (or (i))]{KT}. 
\end{rem}
%
%%%%%%%%%%%%%%%%%%%%%%%%%%%%%%%%%%%%%%%%%%%%%%%%%%%%%%%%%%%%%%%
\appendix
\def\thesection{B}
\section{Property Related to the Characterization} \label{sec:some properties}
\newcommand{\LW}{\mathrm{LW}}
In this appendix, we show that there is no inclusion relation between $\D_\Delta$ and $\D_{\Delta'}$ (Corollary \ref{cor:cap}). 
\begin{prop}[] \label{prop:cap}
    % 非可換べき級数環 $\Hoff\llbracket u_1, \ldots, u_s\rrbracket $ 上で
    % 自己同型写像 $\Delta = {\Delta_{u_1}}^{e_1} \circ \cdots \circ {\Delta_{u_s}}^{e_s}$, 
    Let $(e_1, \ldots, e_s), (f_1, \ldots, f_s) \in \mathbb{Z}^s$ with $(e_1, \ldots, e_s) \neq (f_1, \ldots, f_s)$, $\Delta = {\Delta_{u_1}}^{e_1} \circ \cdots \circ {\Delta_{u_s}}^{e_s}$ and 
    $\Delta' = {\Delta_{u_1}}^{f_1} \circ \cdots \circ {\Delta_{u_s}}^{f_s}$. 
    Then, we have the following on $\Hoff\llbracket u_1, \ldots, u_s\rrbracket $: 
    \begin{enumerate}
        \item[(i)] $\mathrm{Im}(\Delta -\tau) \cap \mathrm{Im}(\Delta' -\tau) = \{0\}. $
        \item[(ii)] $\mathrm{Ker}(\Delta -\tau) \cap \mathrm{Ker}(\Delta' -\tau) = (\Q[z])\llbracket u_1, \ldots, u_s\rrbracket , $
    \end{enumerate}
    where $z = x+y$. 
\end{prop}
Reviewing Equation \eqref{eqn:Ker=D}, we obtain the following property from Proposition \ref{prop:cap} (ii). 
\begin{cor} \label{cor:cap}
    \[
        \D_{\Delta} \cap \D_{\Delta'} = (\Q[z])\llbracket u_1, \ldots, u_s\rrbracket. 
    \]
    In particular, 
    \[
        \D_{\Delta}^0 \cap \D_{\Delta'}^0 = \Q\llbracket u_1, \ldots, u_s\rrbracket. 
    \]
\end{cor}
\begin{rem}\upshape
    Recalling that $\D_\Delta$ obviously includes $(\Q[z])\llbracket u_1, \ldots, u_s\rrbracket$, we find that $\D_\Delta$ and $\D_{\Delta'}$ have no common element, except for trivial elements. That is, there is no inclusion relation between $\D_\Delta$ and $\D_{\Delta'}$. 
\end{rem}
In order to prove Proposition \ref{prop:cap}, we show three lemmas. 
Let $\Q\langle x, z \rangle$ be the non-commutative polynomial algebra over the rational in indeterminates $x$ and $z$, where $z=x+y$. 
First, we introduce the lexicographic order $1 < x < z$ for words of $\Q\langle x, z \rangle$. Here, the empty word is $1$. 
% wordの線形和 $p(x, z) \, (\neq 0)\in \Q\langle x, z \rangle$ を構成する words のうち, 係数が $0$ でない最大のwordを $\LW( p(x, z) )$ で表す. 
For a nonzero polynomial $p(x, z) \in \Q\langle x, z \rangle $, let $\LW( p(x, z) )$ denote the largest word for which the coefficient in $p(x, z)$ is nonzero. 
% このとき, 係数の大小は考慮しない. 
Note that we are not concerned with the values of the coefficients. 
\begin{lem} \label{lem:do}
    % 任意の正整数 $n$, $\Q\langle x, z \rangle - \Q[z]$ の元であるwords $w, w_1, w_2$ に対して, 以下が成り立つ: 
    Let $n$ be a positive integer. 
    \begin{enumerate}
        \item[(i)] For a word $w \in \Q\langle x, z \rangle$, we have 
        \[
            \LW(\partial_n(xw)) = xz^{n}w. 
        \]
        \item[(ii)] For words $w_1, w_2 \in \Q\langle x, z \rangle - \Q[z]$ with $w_1 < w_2$, we have 
        \[
            \LW(\partial_n(w_1)) < \LW(\partial_n(w_2)). 
        \]
    \end{enumerate}
\end{lem}
\begin{proof}
    % (i)の等式について, 写像 $\partial_n$ の定義から, $\partial_n(x)=xz^{n-1}(z-x)$, $\partial_n(z)=0$ であるので, 
    From the definition of $\partial_n$, it is obvious that $\partial_n(x)=xz^{n-1}(z-x)$ and $\partial_n(z)=0$. 
    In the case of $w \in \Q[z]$, (i) is valid. 
    In the other case, it suffices to check $\LW(x\partial_n(w)) < \LW(\partial_n(x)w)$. 
    We see 
    \begin{equation} \label{eqn:xw}
        \LW(\partial_n(x)w) = \LW(xz^{n-1}(z - x)w) = xz^{n}w. 
    \end{equation}
    % がわかる. 
    % 一方, $x\partial_n(w)$について, 
    For the term $x\partial_n(w)$, since there exists a non-negative integer $m$ such that $z^{m}xw' = w$ ($w' \in \Q \langle x, z \rangle$), we have
    % $w = z^{m}xw'$ ($m$ は非負整数, $w' \in \Q \langle x, z \rangle$) と表せて, 
    \begin{align*}
        x\partial_n(w)
        &= x\partial_n(z^{m} xw') \\
        &= xz^{m} \partial_n(xw')\\
        &= xz^{m} \partial_n(x)w' + xz^{m}x\partial_n(w') \\
        &= xz^{m} xz^{n-1} (z -x ) w' + xz^{m}x\partial_n(w'). 
    \end{align*}
    %と変形される. (上の式は, $m=0$でも成り立つ. )
    Note that the above equation holds even if $m=0$. 
    % $\LW(\partial_n(x)w) = xz^{m+n}xw'$ であることから, 左から$m+2$ 番目の文字と比較することで $\LW(\partial_n(x)w) > \LW(x\partial_n(w))$ がわかる. 
    By comparing the $m+2$th letter from the left, $\LW(x\partial_n(w)) < \LW(\partial_n(x)w)$ follows. 
    % よって, 
    Therefore, we obtain
    \begin{align*}
        \LW(\partial_n(xw)) 
        &= \LW(\partial_n(x)w + x\partial_n(w)) 
        = \LW(\partial_n(x)w) 
        = xz^{n}w. 
    \end{align*}
    %を得る. 
    %
    % 次に, (ii)の不等式を3つの場合わけで示す. 
    Next, we prove (ii) in three cases. 
    % ひとつめの場合わけとして, $w_1 = xw'_1$, $w_2 = xw'_2$ ($w'_1 < w'_2$)と表せるとき, (i)から
    First, in the case of $w_1 = xw'_1$ and $w_2 = xw'_2$ with $w'_1 < w'_2$, we obtain the following using (i): 
    \[
        \LW(\partial_n(xw'_1)) = xz^{n}w'_1 < xz^{n}w'_2 = \LW(\partial_n(xw'_2)). 
    \]
    % を得る. 
    % ふたつめの場合わけとして, $w_1 = xw'_1$, $w_2 = zw'_2$ と表せるとき, 
    Second, in the case of $w_1 = xw'_1$ and $w_2 = zw'_2$, we obtain
    \[
        \LW(\partial_n(xw'_1)) = xz^nw'_1 < z \LW(\partial_n(w'_2)) = \LW(\partial_n(zw'_2)). 
    \]
    % を得る. 
    % みっつめの場合わけとして, $w_1 = z^m xw'_1$, $w_2 = z^m w'_2$ ($m$ は正整数, $xw'_1 < w'_2$)と表せるとき, 
    Finally, in the case of $w_1 = z^m xw'_1$ and $w_2 = z^m w'_2$, where $m$ is a positive integer and $xw'_1 < w'_2$, then
    \[
        \LW(\partial_n(z^m xw'_1)) = z^m \LW(\partial_n(xw'_1)) 
    \]
    and 
    \[
        \LW(\partial_n(z^m w'_2)) = z^m \LW(\partial_n(w'_2)). 
    \]
    % がわかるので, 
    % $\LW(\partial_n(xw'_1)) < \LW(\partial_n(w'_2))$ であればよく, すでに確認した場合分けから示される. 
    Now, applying other cases, we obtain $\LW(\partial_n(xw'_1)) < \LW(\partial_n(w'_2))$, and so (ii) is proved. 
\end{proof}
\begin{lem} \label{lem:partial=0}
    % $\Hoff\llbracket u_1, \ldots, u_s\rrbracket $ 上の写像 $\partial_n$ ($n$ は正整数) について, 次が成り立つ: 
    For any positive integer $n$, we have the following on $\Hoff\llbracket u_1, \ldots, u_s \rrbracket$: 
    \[
        \mathrm{Ker}(\partial_n) = (\Q[z])\llbracket u_1, \ldots, u_s \rrbracket. 
    \]
\end{lem}
\begin{proof}
    % 集合 $\mathrm{Ker}(\partial_n)$ が $(\Q[z])\llbracket u_1, \ldots, u_s\rrbracket $ を包含することは明らか. 逆の包含関係を示すために, $\mathrm{Ker}(\partial_n)$ の元を, $w = \sum_{i \geq 0} w_i$ ($w_i \in \Hoff\llbracket u_1, \ldots, u_s\rrbracket $, $\mathrm{deg}_{x,y}(w_i) = i$)とおく. 
    It is trivial that $\mathrm{Ker}(\partial_n) \supset(\Q[z]) \llbracket u_1, \ldots, u_s \rrbracket$. 
    We check that the inclusion relation $\mathrm{Ker}(\partial_n) \subset (\Q[z])\llbracket u_1, \ldots, u_s\rrbracket$. 
    By using $\Hoff \simeq \Q \langle x, z \rangle$, the elements of $\Hoff \llbracket u_1, \ldots, u_s\rrbracket$ can be discussed as elements of $(\Q \langle x, z \rangle) \llbracket u_1, \ldots, u_s\rrbracket$. 
    Any element can be expressed by $\sum_{i \geq 0} w_i$ satisfying $\mathrm{deg}_{x,z}(w_i) = i$ unless $w_i = 0$, and so can each element of $\mathrm{Ker}(\partial_n)$. 
    % 写像 $\partial_n$ は $x$, $y$ についての斉次関数であることを考慮すると, $\partial_n(w) = 0$ である必要十分条件は, 
    % 各次数 $i$ において $\partial_n(w_i)=0$ である. 
    % よって, 任意の次数 $i$ において, $\partial_n(w_i)=0$ ならば $w_i \in (\Q[z])\llbracket u_1, \ldots, u_s\rrbracket $ であることを示せばよい. 次数 $i$ についての帰納法で証明する. 
    Given $\sum_{i \geq 0} w_i \in \mathrm{Ker}(\partial_n)$, we prove $w_i \in (\Q[z])\llbracket u_1, \ldots, u_s\rrbracket$ 
    % 次数 $i$ についての帰納法で証明する. 
    by induction on $i$. 
    It is obvious that $w_0 \in (\Q[z])\llbracket u_1, \ldots, u_s\rrbracket $, and that $\partial_n(w_i)=0$ for any positive integer $i$. 
    % まず $w_0$ は $\Q\llbracket u_1, \ldots, u_s\rrbracket $ の元なので, $\partial_n$ の定義から明らかに成り立つ. 
    % 次に, 次数 $i-1$ ($i \geq 1$)において, 主張が成り立つとする. 
    % 自明な同型 $\Q \langle x, y \rangle \simeq \Q \langle x, z \rangle$ から, 次数 $i$ の多項式 $w_i$ は, 
    Assume this holds in the case of $i-1$ for some $i \geq 1$. 
    % 自明な同型 $\Q \langle x, y \rangle \simeq \Q \langle x, z \rangle$ から, 次数 $i$ の多項式 $w_i$ は, 
    There are $c_{a_1, \ldots, a_i} \in \Q\llbracket u_1, \ldots, u_s\rrbracket $ such that 
    \[
        w_i = \sum_{a_1, \ldots, a_i \in \{x, z\}}c_{a_1, \ldots, a_i} a_1 \cdots a_i. 
    \]
    % と表せる($c_{a_1, \ldots, a_i} \in \Q\llbracket u_1, \ldots, u_s\rrbracket $). 
    % 両辺に $\partial_n$ を作用させると, 
    By applying $\partial_n$ to the above expression and using $\partial_n(zw)=z\partial(w)$ ($w \in \Hoff$), we have 
    \begin{align*}
        \partial_n(w_i) 
        &= \partial_n \left(\sum_{a_1, \ldots, a_i \in \{x, z\}}c_{a_1, \ldots, a_i} a_1 \cdots a_i \right)\\
        &= \partial_n \left( \sum_{a_2, \ldots, a_i \in \{x, z\}}c_{x, a_2, \ldots, a_i} \: x a_2 \cdots a_i \right) 
        + z \partial_n \left( \sum_{a_2, \ldots, a_i \in \{x, z\}}c_{z, a_2, \ldots, a_i} \: a_2 \cdots a_i \right). 
    \end{align*}
    % を得る. 
    % 第1項の1文字目は $x$, 第2項の1文字目は $z$ であるので, $\partial_n(w_i) = 0$ であるためには以下が成り立たなくてはならない: 
    The first letters of the first and second terms are $x$ and $z$, respectively. Since the first letters of the two terms are different, we find that $\partial_n(w_i) = 0$ if and only if 
    \[
        \begin{cases}
            \partial_n \left( \sum_{a_2, \ldots, a_i \in \{x, z\}}c_{x, a_2, \ldots, a_i} \: x a_2 \cdots a_i \right) =0, \\
            z \partial_n \left( \sum_{a_2, \ldots, a_i \in \{x, z\}}c_{z, a_2, \ldots, a_i} \: a_2 \cdots a_i \right) = 0. 
        \end{cases}
    \] 
    % ふたつめの式に帰納法の仮定を適用すると, $\sum_{a_2, \ldots, a_i \in \{x, z\}}c_{z, a_2, \ldots, a_i} a_2 \cdots a_i \in (\Q[z])\llbracket u_1, \ldots, u_s\rrbracket $ がわかり, すなわち, $c_{z, \ldots, z}$ 以外の係数は $0$ であることを意味する. 
    By the induction hypothesis applied to the second equation, we obtain 
    \[
        \sum_{a_2, \ldots, a_i \in \{x, z\}}c_{z, a_2, \ldots, a_i} a_2 \cdots a_i \in (\Q[z])\llbracket u_1, \ldots, u_s \rrbracket. 
    \]
    This means that $c_{z, a_2, \cdots, a_i} = 0$ unless $a_2 = \ldots = a_i = z$. 
    % 一方, ひとつめの式において, 係数のうち少なくとも1つは $c_{x, a_2, \ldots, a_i} \neq 0$ であると仮定して, 背理法で証明する. 
    On the other hand, we derive a contradiction from the first equation to prove $c_{x, a_2, \ldots, a_i} = 0$. 
    Suppose that at least one of $c_{x, a_2, \ldots, a_i}$ is nonzero. 
    % このとき, $\sum_{a_2, \ldots, a_i \in \{x, z\}} c_{x, \ldots, a_i} x a_2 \cdots a_i$
    % を構成する words のうち, 係数が $0$ でない辞書式順序が最大のwordを $x b_2 \cdots b_i$ とする. 
    Let words $v_1 < \ldots < v_l$ with degree $i-1$ for a positive number $l$, then $\LW(\partial_n(xv_1)) < \ldots < \LW(\partial_n(xv_l))$ follows from Lemma \ref{lem:do}. 
    Now, there exists the largest word of $x a_2 \cdots a_i$ such that $c_{x, a_2, \ldots, a_i}$ is nonzero. Let $x b_2 \cdots b_i$ be this word. 
    % すると, 補題\ref{lem:do} (ii)と $x b_2 \cdots b_i$ の最大性から $\LW(\partial_n(xb_2 \cdots b_i))$ をもつ
    % 他の $\partial_n(x a_2 \cdots a_i)$ は存在しないため, 多項式の和は $0$ とならず矛盾. 
    Since $\LW(\partial_n(x b_2 \cdots b_i))$ is larger than any other $\LW(\partial_n(x a_2 \cdots a_i))$, we have that 
    \[
        \LW\left(\partial_n \left( \sum_{a_2, \ldots, a_i \in \{x, z\}}c_{x, a_2, \ldots, a_i} \: x a_2 \cdots a_i \right)\right) = \LW(\partial_n(x b_2 \cdots b_i)). 
    \]
    However, it is not zero, which is a contradiction. 
    %
    % したがって, $c_{x, a_2, \ldots, a_i}=0$ がわかり, $w_i = c_{z, \ldots, z} z^i$ を得る. 
    % 以上から任意の $i$ について $w_i \in (\Q[z])\llbracket u_1, \ldots, u_s\rrbracket $ が成り立つので, $w \in (\Q[z])\llbracket u_1, \ldots, u_s\rrbracket $ であることが示された. 
    Therefore, we obtain that $c_{a_1,\ldots, a_i}=0$, unless $a_1 =  \cdots = a_i = z$, i.e., $w_i \in (\Q[z])\llbracket u_1, \ldots, u_s\rrbracket$. 
\end{proof}
\begin{lem} \label{lem:Delta=id}
    For integers $e_1, \ldots, e_s$, let $\Delta = {\Delta_{u_1}}^{e_1} \circ \cdots \circ {\Delta_{u_s}}^{e_s}$ on $\Hoff\llbracket u_1, \ldots, u_s\rrbracket $. If $(e_1, \ldots, e_s) \neq (0, \ldots, 0)$, then we have
    \[
        \mathrm{Ker}(\Delta - \mathrm{id}) = (\Q[z])\llbracket u_1, \ldots, u_s\rrbracket . 
    \]
\end{lem}
\begin{proof}
    % まず, $\mathrm{Ker}(\Delta - \mathrm{id}) \supset (\Q[z])\llbracket u_1, \ldots, u_s\rrbracket $ は明らか. 逆の包含関係を示す. 
    It is trivial that $\mathrm{Ker}(\Delta - \mathrm{id}) \supset (\Q[z])\llbracket u_1, \ldots, u_s\rrbracket$, and so we check $\mathrm{Ker}(\Delta - \mathrm{id}) \subset (\Q[z])\llbracket u_1, \ldots, u_s\rrbracket $. 
    % 不定元 $u_1, \ldots, u_s$ の対称性から, 写像 $\Delta = {\Delta_{u_1}}^{e_1} \circ \cdots \circ {\Delta_{u_p}}^{e_p} \circ {\Delta_{u_{p+1}}}^{-e_{p+1}} \circ \cdots \circ {\Delta_{u_{p+q}}}^{-e_{p+q}}$ ($e_1, \ldots, e_{p+q} > 0$, $1 \leq p+q \leq s$) としても一般性は失われないので, 次を示せばよい: 
    We may assume $\Delta = {\Delta_{u_1}}^{e_1} \circ \cdots \circ {\Delta_{u_p}}^{e_p} \circ {\Delta_{u_{p+1}}}^{-e_{p+1}} \circ \cdots \circ {\Delta_{u_{p+q}}}^{-e_{p+q}}$ ($e_1, \ldots, e_{p+q} > 0$, $1 \leq p+q \leq s$) without loss of generality. Then, by replacing $w \in \Hoff\llbracket u_1, \ldots, u_s \rrbracket$ with ${\Delta_{u_{p+1}}}^{e_{p+1}} \circ \cdots \circ {\Delta_{u_{p+q}}}^{e_{p+q}}(w)$, it suffices to prove the following: 
    % 非可換べき級数環 $\Hoff\llbracket u_1, \ldots, u_s\rrbracket $ 上の自己同型写像 ${\Delta_{u_1}}^{e_1} \circ \cdots \circ {\Delta_{u_p}}^{e_p} \circ {\Delta_{u_{p+1}}}^{-e_{p+1}} \circ \cdots \circ {\Delta_{u_{p+q}}}^{-e_{p+q}}$ 
    \begin{equation}
        \mathrm{Ker}\left( {\Delta_{u_1}}^{e_1} \circ \cdots \circ {\Delta_{u_p}}^{e_p} - {\Delta_{u_{p+1}}}^{e_{p+1}} \circ \cdots \circ {\Delta_{u_{p+q}}}^{e_{p+q}} \right) \subset (\Q[z])\llbracket u_1, \ldots, u_s\rrbracket . 
    \end{equation}
    % 各写像を展開し, 不定元 $u_1, \ldots, u_s$ の次数が$2$未満の項に着目すると, 
    Now, by focusing on the part of total degree $1$ in $u_1, \ldots, u_s$, each ${\Delta_{u_1}}^{e_1}$ is expanded as follows:
    \begin{align*}
        &{\Delta_{u_1}}^{e_1} \circ \cdots \circ {\Delta_{u_p}}^{e_p} - {\Delta_{u_{p+1}}}^{e_{p+1}} \circ \cdots \circ {\Delta_{u_{p+q}}}^{e_{p+q}} \\ 
        &= \prod_{i=1}^{p} \left( \mathrm{id} + \binom{e_i}{1}\theta_1u_i + \binom{e_i}{1}\theta_2{u_i}^{2} + \binom{e_i}{2} {\theta_{1}}^{2}{u_i}^{2} + \cdots \right) \\
        &\quad - \prod_{j=1}^{q} \left( \mathrm{id} + \binom{e_{p+j}}{1}\theta_1 u_{p+j} + \binom{e_{p+j}}{1} \theta_2{u_{p+j}}^{2} + \binom{e_{p+j}}{2} {\theta_{1}}^{2}{u_{p+j}}^{2} + \cdots \right) \\
        &= \mathrm{id} + \sum_{i = 1}^{p} e_i \theta_1 {u_i} + (\text{the total degree $2$ or higher part}) \\
        &\quad - \mathrm{id} - \sum_{j = 1}^{q} e_{p+j} \theta_1 {u_{p+j}} + (\text{the total degree $2$ or higher part}) \\
        &= \sum_{i = 1}^{p} e_i \theta_1 {u_i} -\sum_{j = 1}^{q} e_{p+j} \theta_1 {u_{p+j}} + (\text{the total degree $2$ or higher part}), 
    \end{align*}
    %を得る. ただし, $\prod$ は写像の合成を表す. 
    where $\prod$ implies the composition of maps. 
    For an element $w$ in the kernel, let $m_0$ be the lowest degree of all words with respect to $x$ and $y$. 
    Setting $w = \sum_{m \geq m_0} w_{m}$ ($w_{m} \in \Hoff\llbracket u_1, \ldots, u_s\rrbracket $) with $\mathrm{deg}_{x, y}(w_{m}) = m$ unless $w_m = 0$, we have 
    \[
        \sum_{i = 1}^{p} e_i \sum_{m \geq m_0} \theta_1(w_{m}) {u_i} 
        -\sum_{j = 1}^{q} e_{p+j} \sum_{m \geq m_0} \theta_1(w_{m}) {u_{p+j}} 
        + (\text{the total degree $2$ or higher part}) = 0. 
    \]
    %である. 
    % 不定元 $x, y$ についての次数ごとに取り出しても $0$ に等しいので, 最小次数項を取り出すと, 
    Focusing the lowest-degree part in $x$ and $y$, we obtain 
    \[
        \sum_{i = 1}^{p} e_i \theta_1(w_{m_0}) {u_i} 
        -\sum_{j = 1}^{q} e_{p+j} \theta_1(w_{m_0}) {u_{p+j}} = 0. 
    \]
    % であり, 整理すると, 
    The above equation is transformed as follows:
    \[
        \theta_1(w_{m_0}) \left(
            \sum_{i = 1}^{p} e_i u_i -\sum_{j = 1}^{q} e_{p+j} u_{p+j}
        \right) = 0. 
    \]
    % がわかる. 
    % 形式的べき級数環 $\Hoff\llbracket u_1, \ldots, u_s\rrbracket $ は整域, かつ, $(e_1, \ldots, e_s) \neq (0, \ldots, 0)$ なので, $\theta_1(w_{m_0})=0$ でなくてはならない. 
    Since $(e_1, \ldots, e_s) \neq (0, \ldots, 0)$ and $\Hoff\llbracket u_1, \ldots, u_s\rrbracket $ is an integral domain, we see $\theta_1(w_{m_0})=0$. 
    % 補題\ref{lem:partial=0}, および$\theta_1 = \partial_1$ を適用すると, 
    % $w_{m_0} \in (\Q[z])\llbracket u_1, \ldots, u_s\rrbracket $ がわかる. 
    By using Lemma \ref{lem:partial=0} and $\theta_1 = \partial_1$, 
    we obtain that $w_{m_0} \in (\Q[z])\llbracket u_1, \ldots, u_s\rrbracket$ holds. 
    Now, it is easy to see that 
    \[
        (\Delta - \mathrm{id})(w-w_{m_0}) = 0. 
    \]
    % が成り立つ. 
    % 多項式 $w - w_{m_0}$ の最小次数は $m_0$ よりも大きく, 
    % $w-w_{m_0}$ についても同様の議論が繰り返せて, 
    % $w \in (\Q[z])\llbracket u_1, \ldots, u_s\rrbracket $ を得る. 
    The lowest degree of all words appearing in $w - w_{m_0}$ is higher than $m_0$, 
    so we obtain $w \in (\Q[z])\llbracket u_1, \ldots, u_s\rrbracket $ by applying the same argument to $w-w_{m_0}$ recursively. 
\end{proof}
Proposition \ref{prop:cap} is proved by using Lemma \ref{lem:Delta=id}. 
\begin{proof}[Proof of Proposition \ref{prop:cap}]
    % (i)について, 補題\ref{lem:Ker=Im} (i)により, 
    We prove Proposition \ref{prop:cap} (i). By applying Lemma \ref{lem:Ker=Im} (i) to the left-hand side, it suffices to check
    \[
        \mathrm{Ker}(\Delta^{-1} + \tau) \cap \mathrm{Ker}(\Delta'^{-1} + \tau) = \{0\}. 
    \]
    % を示せばよい. 
    % 任意の $w \in \mathrm{Ker}(\Delta^{-1} + \tau) \cap \mathrm{Ker}(\Delta'^{-1} + \tau)$ について, 
    Suppose $w \in \mathrm{Ker}(\Delta^{-1} + \tau) \cap \mathrm{Ker}(\Delta'^{-1} + \tau)$, then we see 
    \[
        (\Delta^{-1} + \tau)(w) = (\Delta'^{-1} + \tau)(w) = 0. 
    \]
    % が成り立つので, 
    Therefore, we have
    \[
        \Delta(w) = -\tau(w) = \Delta'(w). 
    \]
    % を得る. 
    % 両辺に $\Delta'^{-1}$ を作用させることで, $\Delta'^{-1} \circ \Delta(w) = w$ を得るので, 
    Since it is easy to see that $\Delta'^{-1} \circ \Delta(w) = w$, we obtain $w \in (\Q[z])\llbracket u_1, \ldots, u_s\rrbracket $ from Lemma \ref{lem:Delta=id}. 
    Reviewing $\Delta(z)=\tau(z)=z$, we have
    \[
        (\Delta^{-1} + \tau)(w)=2w. 
    \]
    % このとき, $w \in \mathrm{Ker}(\Delta^{-1} + \tau)$ であることから, 
    Then, $w=0$ follows from $w \in \mathrm{Ker}(\Delta^{-1} + \tau)$. 
    % (ii)も同様に証明されて, 任意の $w \in \mathrm{Ker}(\Delta - \tau) \cap \mathrm{Ker}(\Delta' - \tau)$ について, 
    The proof of Proposition \ref{prop:cap} (ii) is similar and so will be omitted. 
    % $\Delta'^{-1} \circ \Delta(w) = w$ であるので, 
    % 補題\ref{lem:Delta=id}により, $w \in (\Q[z])\llbracket u_1, \ldots, u_s\rrbracket $ を得る. 
\end{proof}
%
%%%%%%%%%%%%%%%
% 謝辞
%%%%%%%%%%%%%%%
\section*{Acknowledgments}
I would like to express my greatest appreciation to my advisor Professor Yasuo Ohno for giving me many pieces of advice. 
I also thank past and present members of our laboratory. 
This work was supported by the WISE Program for AI Electronics (Tohoku University) and JST SPRING (Grant Number JPMJSP2114). 
%
%%%%%%%%%%%%%%%%%%%%%%%%%%%%%%%%%%
%参考文献
%%%%%%%%%%%%%%%%%%%%%%%%%%%%%%%%%%

%
\address{Mathematical Institute, Tohoku University, 6-3 Aoba, Aramaki, Aoba-ku, Sendai 980-8578, Japan}

\email{E-mail address: aiki.kimura.t2@dc.tohoku.ac.jp}

\begin{thebibliography}{99}
    \bibitem{Hoffman} M.~E.~Hoffman, \textit{The algebra of multiple harmonic series}, J.~Algebra \textbf{194} (1997), no.~2, 477--495.
    \bibitem{IKZ} K.~Ihara, M.~Kaneko, and D.~Zagier, \textit{Derivation and double shuffle relations for multiple zeta values}, Compos. Math. \textbf{142} (2006), no.~2, 307--338.
    \bibitem{Kaji} J.~Kajikawa, \textit{Duality and double shuffle relations of multiple zeta values}, J.~Number Theory \textbf{121} (2006), no.~1, 1--6.
    \bibitem{KT} N.~Kawasaki and T.~Tanaka, \textit{On the duality and the derivation relations for multiple zeta values}, Ramanujan~J. \textbf{47} (2018), no.~3, 651--658.
    \bibitem{Li} Z.~Li, \textit{Derivation relations and duality for the sum of multiple zeta values}, Funct. Approx. Comment. Math. \textbf{58} (2018), no.~2, 215--220.
    \bibitem{Zagier} D.~Zagier, \textit{Values of zeta functions and their applications}, First European Congress of Mathematics, Vol.~II (Paris, 1992), 497--512, Progr. Math., \textbf{120}, Birkh\"{a}user, Basel, 1994.
\end{thebibliography}
\end{document}